\documentclass[journal]{IEEEtran}
\usepackage{cite}

\ifCLASSINFOpdf
   \usepackage[pdftex]{graphicx}
   \graphicspath{{./pdf/}{./jpeg/}}
   \DeclareGraphicsExtensions{.pdf,.jpeg,.png}
\else
   \usepackage[dvips]{graphicx}
   \usepackage{epstopdf}
   \graphicspath{{./eps/}}
   \DeclareGraphicsExtensions{.eps}
\fi
\usepackage{subfig}

\usepackage{amsmath,amsfonts,amssymb,amsthm}
\usepackage[algo2e,vlined,ruled]{algorithm2e}

\newtheorem{theorem}{Theorem}
\newtheorem{lemma}{Lemma}
\newtheorem{assumption}{Assumption}
\newtheorem{remark}{Remark}

\newcommand{\R}{\mathbb{R}}

\newcommand{\cC}{\mathcal{C}}
\newcommand{\cE}{\mathcal{E}}

\newcommand{\cK}{\mathcal{K}}
\newcommand{\cG}{\mathcal{G}}
\newcommand{\cI}{\mathcal{I}}
\newcommand{\cJ}{\mathcal{J}}
\newcommand{\cL}{\mathcal{L}}

\newcommand{\cN}{\mathcal{N}}
\newcommand{\cS}{{\mathcal{S}}}
\newcommand{\cO}{{\mathcal{O}}}
\newcommand{\cX}{\mathcal{X}}

\newcommand{\Reg}{\mathrm{Reg}}
\newcommand{\Vio}{\mathrm{Vio}}
\newcommand{\iprod}[2]{\langle #1, #2 \rangle}
\DeclareMathOperator*{\argmin}{arg\,min}
\newcommand{\be}{\begin{equation}}
\newcommand{\ee}{\end{equation}}
\usepackage{xcolor}

\begin{document}
%
\title{Augmented Lagrangian Methods for Time-varying Constrained Online Convex Optimization}
%
%
%

\author{Haoyang Liu,  Xiantao Xiao, and Liwei Zhang
\thanks{This work was supported in part by  the National Natural Science Foundation of China (No. 11731013, No. 11971089  and No. 11871135).}
\thanks{The authors are with
the School of Mathematical Sciences, Dalian University of Technology, 116023, Dalian,
China (e-mail: hyliu@mail.dlut.edu.cn; xtxiao@dlut.edu.cn; lwzhang@dlut.edu.cn).}}

\maketitle

\begin{abstract}
In this paper, we consider online convex optimization (OCO) with time-varying loss and constraint functions. Specifically, the decision maker chooses sequential decisions  based only on past information, meantime the loss and constraint functions are revealed over time. We first develop a class of model-based augmented Lagrangian methods (MALM) for time-varying functional constrained OCO (without feedback delay). Under standard assumptions, we establish sublinear regret and sublinear constraint violation of MALM. Furthermore, we extend MALM to deal with time-varying functional constrained OCO with delayed feedback, in which the feedback information of  loss and constraint functions is revealed to decision maker with delays. Without additional assumptions, we also establish  sublinear regret and sublinear constraint violation for the delayed version of MALM. Finally, numerical results for several examples of constrained OCO including online network resource allocation, online logistic regression and online quadratically constrained quadratical program are presented to demonstrate the efficiency of the proposed algorithms.
\end{abstract}


%
\IEEEpeerreviewmaketitle

\section{Introduction}
\label{sec:introduction}
\IEEEPARstart{I}{n} recent years, online convex optimization (OCO) has attracted more and more attention, due to its broad applicability in a wide variety of fields including
machine learning, signal processing and control \cite{Shai2011,Hazan2015,HSLZ2021}.
In online optimization, the player or decision maker must make decisions sequentially based only on past environment feedback. In this setting, it is impossible to design an algorithm to obtain the exact optimal solution due to lack of needed information. Therefore, the standard measurement of the performance of OCO algorithms is \textit{regret} for unconstrained or time-invariant constrained OCO, together with \textit{constraint violation} for time-varying constrained OCO.

For OCO without functional constraint or with time-invariant constraints, a number of computationally efficient algorithms have been studied. We only name a few of them in the following. In the pioneering work \cite{Zi2003}, Zinkevich developed an online projected gradient descent algorithm for minimizing a time-varying loss function constrained to a simple stationary set, and proved that this algorithm possesses a sublinear regret of $\cO(\sqrt{T})$, where $T$ is the length of the total time horizon. Later, a simple example was constructed in \cite{HAK2007} to show that $\cO(\sqrt{T})$ regret is tight for an arbitrary sequence of differentiable convex loss functions and better regret is possible under the assumption that each loss function is strongly convex. In \cite{LTS2021}, a second-order approach named online Newton's method was proposed and shown to outperform online gradient descent algorithm.    OCO with long term constraints, in which the constraints are time-invariant, was first studied in \cite{MRYang2012} and later solved in \cite{YN2020} by a low complexity algorithm with $\cO(\sqrt{T})$ regret and $\cO(1)$ constraint violation.

In this paper, we focus on constrained OCO with time-varying constraints. The authors in \cite{MTY2009} considered a problem that minimizes a time-varying convex objective function subject to the time average of a single time-varying constraint function being less than or equal to 0. A continuous time constrained OCO with long term time-varying constraints was studied in \cite{PR2017}. For constrained OCO with time-varying constraints, a modified online saddle-point (MOSP) method  was proposed in \cite{CLG2017} and shown to possess $\cO(T^{\frac{2}{3}})$ bounds of regret and constraint violation. An online virtual-queue based algorithm (simply denoted by NY here) was developed in \cite{NYu2017} and proved that the improved $\cO(\sqrt{T})$ bounds of regret and constraint violation are attained. An online saddle-point algorithm (simply denoted by CL here) associated with a modified Lagrangian was studied in \cite{CLiu2019} for constrained OCO with time-varying constraints and bandit feedback. After equivalent reformulations,  MOSP, NY and CL can all be viewed as online Lagrangian-based primal-dual methods, see Section \ref{subsec:nra} for a  detailed discussion.
An online alternating direction method of multipliers was proposed in \cite{ZDH2021} for OCO with linear time-varying constraints.
The  prediction-correction interior-point method was investigated in \cite{FPPR2018} for time-varying convex optimization, which is different from the setting of OCO since it needs the information of the current objective and constraint functions. Another related research line is distributed OCO with time-varying constraints \cite{CB2021a,Cb2021b,YLYXCJ2021}, which is beyond the scope of this paper.

Most of the aforementioned algorithms for constrained OCO with time-varying constraints can be cast into the family of online first-order algorithms. It is well known in the community of optimization and beyond  that although the iteration in first-order algorithms is computationally cheap and some of them perform well for specific structured problems, there are plenty of practical experiences and theoretical evidences of their convergence difficulties and instability with respect to the choice of parameters and stepsizes. For convex programming, augmented Lagrangian method is known to converge asymptotically superlinearly \cite{Rockafellar76B}. Moreover,   the success of augmented Lagrangian methods for various functional constrained optimization problems was witnessed in past decades. For instance, a Newton-CG augmented Lagrangian method for semidefinite programming (SDP) was considered in \cite{ZST2010} and shown to be very efficient even for large-scale SDP problems.
 Several augmented Lagrangian-based algorithms for distributed optimization were proposed in \cite{CDZ2015,CZ2016,MK2021}. In \cite{ZZWX2022}, a stochastic linearized augmented Lagrangian method with proximal term was developed for solving stochastic optimization with expectation constraints and demonstrated to outperform several existing stochastic first-order algorithms. Inspired by these observations, the aim of this paper is to study efficient online augmented Lagrangian-type methods for constrained OCO with time-varying constraints. To the best of our knowledge, this is still limited in the literature.

 The main contributions of this paper are summarized as follows.
 \begin{itemize}
 \item We propose a class of model-based augmented Lagrangian methods (MALM) for time-varying constrained OCO (without feedback delay). One of the notorious disadvantages for traditional augmented Lagrangian methods is that the subproblem is usually computationally intractable. However, in MALM, we show that this issue can be resolved by choose the model properly. Under standard assumptions, we establish $\cO(\sqrt{T})$ sublinear bounds of both regret and constraint violation, which are better than those of MOSP and CL. Nevertheless, in the convergence analysis, unlike MOSP and CL,  the corresponding stepsizes of MALM are chosen to $\alpha=\sqrt{T}$, $\sigma=1/\sqrt{T}$ which are independent of problem parameters. Moreover, numerical results of two examples named online network resource allocation and online logistic regression are presented to verify the efficiency of MALM, particularly compared with MOSP, NY and CL.

 \item In certain applications of OCO, the feedback information of the loss and/or constraint functions is often revealed to decision makers with delay $\tau\geq 0$ due to environment reaction time or computation burden \cite{ZLS2009,JGS2013,PSLW2022}. Time-varying constrained OCO with feedback delays was initially studied in \cite{CZP2021} very recently. The algorithm in \cite{CZP2021}, which can be regarded as a ``delayed" variant of CL, was shown to possess sublinear regret bound $\cO(\sqrt{\tau T})$ and sublinear constraint violation bound $\cO(T^{\frac{3}{4}}\tau^{\frac{1}{4}})$. Motivated by \cite{CZP2021}, we also extend MALM to solve time-varying constrained OCO with delayed feedback. We establish regret bound $\cO(\sqrt{\tau T})$ and a better constraint violation bound $\cO(T^{\frac{1}{2}}\tau^{\frac{3}{2}})$ (since $\tau$  is usually much smaller than $T$). Additionally, unlike \cite{CZP2021}, the delay $\tau$ is not required to be sublinear with respect to $T$. Finally, numerical experiments of online quadratically constrained quadratical program with feedback delays are conducted to demonstrate the advantage of MALM.
\end{itemize}

The rest of this paper is organized as follows. In Section \ref{sec:nodelay},  the formulation of time-varying constrained OCO (without feedback delay) is introduced and the performance analysis of MALM is presented. In Section \ref{sec:delay},  the  constrained OCO with feedback delays is formally described and the convergence of the corresponding MALM is established. Numerical results on three numerical problems are presented in Section \ref{sec:numerical}. Finally, this paper is concluded in Section \ref{sec:conclude}.
\section{Augmented Lagrangian Methods for  Constrained OCO without Feedback Delay}\label{sec:nodelay}
\subsection{Problem Statement}
The constrained OCO problem can be viewed as a multi-round learning process. At each round (or time slot) $t$, the online decision maker or player makes a decision $x_t\in\cC$, where $\cC$ is a set of admissible actions. After that, the player receives  a loss (objective) function $f_t:\cC\mapsto\R$ and a constraint function $g_t:\cC\mapsto \R^p$  which can be used to make the new decision $x_{t+1}$. Given a fixed time horizon $T$, the best fixed decision $x^*$ is defined by the optimal solution to a constrained optimization problem:
\be\label{eq:xstar}
x^*\in\argmin\limits_{x\in\cC}\sum_{t=0}^{T-1}f_t(x),\ \mbox{s.t.}\ g_t(x)\leq 0, t\in [T],
\ee
where $[T]=\{0,1,2,\ldots,T-1\}$. Since the functions $\{f_t,g_t\}_{t\in[T]}$ are unknown to the player beforehand, it is not possible to derive the best decision $x^*$. Instead, the player proposes  a learning algorithm to generate a decision sequence $\{x_t\}_{t\in[T]}$ and  evaluate its performance  by the regret of the objective reduction and the constraint violation defined by
\be\label{eq:regret}
\begin{array}{ll}
\displaystyle\Reg(T):=\sum_{t=0}^{T-1}[f_t(x_t)-f_t(x^*)],\\[10pt]
\displaystyle\Vio^{(i)}(T):=\sum_{t=0}^{T-1}g_t^{(i)}(x_t),\quad i=1,\ldots,p,
\end{array}
\ee
respectively.
Here, $g_t^{(i)}$ is the component of $g_t$ and let us denote that $g_t(x):=[g^{(1)}_t(x),\dots,g_t^{(p)}(x)]^T$. \

In the sequel, let us assume that $\cC\in\cX$ is a closed convex set and $\cX$ is a finite-dimensional real Hilbert space equipped with an inner product $\iprod{\cdot}{\cdot}$ and its induced norm $\|\cdot\|$. For a vector $\lambda\in\R^p$, the notation $\|\lambda\|$  means the Euclidean norm. We make the following  assumptions, which are standard in the literature, see \cite{CLG2017,NYu2017,CLiu2019}.
\begin{assumption}\label{assu:bounded}
The set $\cC$ is bounded. In particular, there exists a constant $D>0$ such that $\|x-y\|\leq D$ for any $x,y\in\cC$.
\end{assumption}
\begin{assumption}\label{assu:subgradient}
The function $f_t$ and $g_t$  are continuous convex and their subgradients are bounded. In particular,
there exist  positive constants $\kappa_f$ and $\kappa_g$  such that $\|u_t\|\leq \kappa_f$ and $\|v^{(i)}_t\|\leq \kappa_g$  for any $x\in\cC$, $u_t\in\partial f_t(x)$ and  $v^{(i)}_t\in\partial g^{(i)}_t(x)$.
\end{assumption}
\begin{assumption}\label{assu:slater}
The Slater's condition holds, i.e., there exist a positive constant $\varepsilon_0$ and a decision $\widehat{x}\in\cC$ such that  $g_t^{(i)}(\widehat{x} )\leq -\varepsilon_0$, $i=1,\ldots,p$.
\end{assumption}
\subsection{Algorithm Development}
In this subsection, we propose a class of model-based augmented Lagrangian methods (MALM) for solving OCO with time-varying loss and constraint functions.

Let $F_{t,x}$ and $G^{(i)}_{t,x}$ be conservative approximations to $f_t$ and $g^{(i)}_t$ at a reference point  $x\in\cC$, respectively, i.e., $F_{t,x}$ and $G^{(i)}_{t,x}$ are convex functions  and satisfy the following assumption.
\begin{assumption}\label{assu:model}
 \begin{itemize}
 \item[(i)]	For all $y\in\cC$,
 $$
 F_{t,x}(y)\leq f_t(y)\ \mbox{and}\ F_{t,x}(x)=f_t(x).
 $$
  \item[(ii)]	For any $i=1,\ldots,p$ and all $y\in\cC$,
 $$
 G^{(i)}_{t,x}(y)\leq g^{(i)}_t(y)\ \mbox{and}\ G^{(i)}_{t,x}(x)=g^{(i)}_t(x).
 $$
 \item[(iii)] The mapping $G_{t,x}(\cdot):=[G^{(1)}_{t,x}(\cdot),\ldots,G^{(p)}_{t,x}(\cdot)]^T$ is bounded on $\cC$, that is, there exists a constant $\nu_g>0$ such that
 $$
 \|G_{t,x}(y)\|\leq\nu_g,\quad\forall y\in \cC.
 $$
 \end{itemize}
\end{assumption}
In \cite{AD2019I}, the functions $F_{t,x}$ and $G^{(i)}_{t,x}$  were named as the \textit{model} of $f_t$ and $g^{(i)}_t$ at $x\in\cC$, respectively. It was also shown that a family of model-based methods for solving stochastic optimization problems enjoy stronger convergence and robustness guarantees than classical approaches, and add almost no extra computational burden. In the following subsection, we shall introduce some concrete examples of model to show the advantage.

At time slot $t$, consider the following approximation optimization problem:
\be\label{eq:approx}
\begin{array}{ll}
\min\limits_{x\in\cC}\ &F_{t,x_t}(x)\\[8pt]
\mbox{s.t.}\ & G_{t,x_t}(x)\leq 0,
\end{array}
\ee
and the corresponding augmented Lagrangian function:
\[
\cL_{t,\sigma}(x,\lambda):=F_{t,x_t}(x)+\frac{1}{2\sigma}\left[\|[\lambda+\sigma G_{t,x_t}(x)]_+\|^2-\|\lambda\|^2\right].
\]
Here, $\lambda\in\R^p$ is the multiplier, $\sigma>0$ is viewed as the penalty parameter, $[a]_+=\max\{a,0\}$ for a scalar $a\in\R$ and the operator $[\cdot]_+$ is conducted componentwise for a vector.

The proposed model-based augmented Lagrangian methods are conducted as follows.
After the decision $x_t$ and the multiplier $\lambda_t$ are determined,  the player receives the functions $f_t$ and $g_t$. Then, a new action $x_{t+1}$ is taken by solving the following minimization problem:
\be\label{eq:MALM-1}
x_{t+1}=\argmin\limits_{x\in\cC}\left\{\cL_{t,\sigma}(x,\lambda_t)+\frac{\alpha}{2}\|x-x_t\|^2\right\},
\ee
where $\alpha>0$ is the parameter of the proximal term $\frac{1}{2}\|x-x_t\|^2$. Furthermore, the player updates the multiplier $\lambda_{t+1}$ by
\be\label{eq:MALM-2}
\lambda_{t+1}=[\lambda_t+\sigma G_{t,x_t}(x_{t+1})]_+.
\ee
The detail of  the model-based augmented Lagrangian methods for constrained OCO is summarized in Algorithm \ref{alg:MALM0}.
\begin{algorithm2e}[htp]
\caption{MALM for constrained OCO}
\label{alg:MALM0}
\lnlset{alg:MALM1}{1}{Initialization: ~Choose an initial action $x_0 \in \cC$ and select parameters $\sigma>0,\alpha>0$.  Set $\lambda_0=0\in\R^p$.}
\\ \vspace{.5ex}
\For{$t=0,1,2,\ldots,T-1$}{
\lnlset{alg:MALM2}{2}{Submit the decision $x_t$.
} \\ \vspace{.5ex}
\lnlset{alg:MALM3}{3}{Receive the functions $f_t$ and $g_t$.
} \\ \vspace{.5ex}
\lnlset{alg:MALM0-4}{4}{Compute the new decision $x_{t+1}$ according to (\ref{eq:MALM-1}).
} \\ \vspace{.5ex}
\lnlset{alg:MALM0-5}{5}{Update the multiplier $\lambda_{t+1}$ according to (\ref{eq:MALM-2}).
}
\\}
\end{algorithm2e}

We present the sublinear regret and constraint violation of  Algorithm \ref{alg:MALM0} for constrained OCO in the following theorem. The results can be recovered from Theorem \ref{th:regret} and Theorem \ref{th:constraint} by taking the delay $\tau=0$. Although Theorem \ref{th:MALM} can be proved separately by a relatively simpler analysis (no need to deal with the delay), we omit the proof here to save space.
\begin{theorem}\label{th:MALM}
Suppose Assumptions \ref{assu:bounded}-\ref{assu:model} hold. Set $\alpha=\sqrt{T}$ and $\sigma=1/\sqrt{T}$. Then, the regret of Algorithm \ref{alg:MALM0} satisfies
\[
\sum_{t=0}^{T-1}[f_t(x_t)-f_t(x^*)]\leq\frac{\kappa_f^2+\nu_g^2+D^2}{2}\sqrt{T},
\]
where $x^*$ is defined in (\ref{eq:xstar}).  Assume further that $T$ is large enough such that $T>p\kappa_g^2$.  Then, for $i=1,\ldots,p$, the constraint violation of Algorithm \ref{alg:MALM0} satisfies
\[
\begin{array}{ll}
\displaystyle\sum_{t=0}^{T-1}g^{(i)}_t(x_t)
\leq[\kappa_2+2\sqrt{p}\kappa_g^2(\nu_g+\kappa_2)]\\[8pt]
\quad\quad\quad\quad\quad+[(2\sqrt{p}\kappa_g^2+1)(\kappa_0+\kappa_1+2\kappa_3)+2\kappa_g\kappa_f]\sqrt{T},\ \end{array}
\]
where $\kappa_0,\kappa_1,\kappa_2,\kappa_3$ are constants given in (\ref{eq:kappa}).
\end{theorem}
\begin{remark}\label{rem:1}
From Theorem \ref{th:MALM}, we obtain
\[
\frac{1}{T}\sum_{t=0}^{T-1}[f_t(x_t)-f_t(x^*)]\leq\cO\left(\frac{1}{\sqrt{T}}\right)
\]
and
\[
\frac{1}{T}\sum_{t=0}^{T-1}g^{(i)}_t(x_t)\leq\cO\left(\frac{1}{\sqrt{T}}\right),\ i=1,\ldots,p,
\]
which show that the time average regret $\frac{\Reg(T)}{T}\leq\cO(1/\sqrt{T})$ and the time average constraint violation $\frac{\Vio^{(i)}(T)}{T}\leq\cO\left(1/\sqrt{T}\right)$. This means that both of $\frac{\Reg(T)}{T}$ and $\frac{\Vio^{(i)}(T)}{T}$ converge to zero with sublinear rate when the time length $T$ goes to infinity. This is the reason why $\Reg(T)=\cO(T^{\beta})$ with $\beta\in(0,1)$ is used in OCO literature to measure the performance of algorithms.

In this paper, we use the metric $\Reg(T)$ and $\Vio^{(i)}(T)$ defined in (\ref{eq:regret}). Nevertheless, other definitions are possible and also used in the literature. For example, in \cite{LDPSM2019} the following regret was adopted:
\[
\widetilde{\Reg}(T):=\sum_{t=0}^{T-1}[f_t(x_t)-f_t(x^{\#})],
\]
where $x^{\#}$ is the best decision with respect to long-term budget constraints as
\[
x^{\#}\in\argmin\limits_{x\in\cC}\sum_{t=0}^{T-1}f_t(x),\ \mbox{s.t.}\ \sum_{t=0}^{T-1}g_t(x)\leq 0.
\]
It is easy to see that $\Reg(T)\leq \widetilde{\Reg}(T)$ and hence $\widetilde{\Reg}(T)$ is a better metric. However, in \cite{MTY2009} (and \cite{LDPSM2019}) it was shown that the sublinear bounds of $\widetilde{\Reg}(T)$ and $\Vio^{(i)}(T)$ are impossible to attain simultaneously.

In \cite{CLG2017}, a dynamic regret was considered as follows:
\[
\Reg_{\textit{d}}(T):=\sum_{t=0}^{T-1}[f_t(x_t)-f_t(x_t^{*})],
\]
where $\{x^*_t\}$ is a sequence of best dynamic decisions given by
\be\label{eq:dynamic-decision}
x^{*}_t\in\argmin\limits_{x\in\cC}f_t(x),\ \mbox{s.t.}\ g_t(x)\leq 0.
\ee
Again, we have $\Reg(T)\leq \Reg_{\textit{d}}(T)$. But, $\Reg_{\textit{d}}(T)$ is seldom used in the literature, even for unconstrained OCO \cite{Hazan2015}.

In \cite{CLG2017}, a dynamic constraint violation was also considered:
\[
\Vio_{\textit{d}}(T):=\left\|\left[\sum_{t=0}^{T-1}g_t(x_t)\right]_+\right\|,
\]
which is not stronger than $\Vio^{(i)}(T)$.  Since $\Vio^{(i)}(T)\leq\cO(\sqrt{T})$ must imply \[\left[\sum_{t=0}^{T-1}g^{(i)}_t(x_t)\right]_+\leq\cO(\sqrt{T}),\]
and hence
\[
\Vio_{\textit{d}}(T)=\sqrt{\sum_{i=1}^p\left[\sum_{t=0}^{T-1}g^{(i)}_t(x_t)\right]^2_+}\leq \cO(\sqrt{T}).
\]

\end{remark}
\begin{remark}\label{rem:2}
In \cite{CLG2017,NYu2017,CLiu2019}, the algorithms MOSP, NY and CL were proposed for solving OCO with time-varying loss and constraint functions, respectively. From Section \ref{subsec:nra}, it can be shown that these three algorithms are very similar especially for linear constrained OCO.
 All of these algorithms can be equivalently written as the primal-dual gradient methods associated with certain modified Lagrangian functions, and hence  cast into the family of online first-order methods. Meanwhile, MALM is an online version of proximal method of multipliers \cite{Rockafellar76B}, which is a type of online augmented Lagrangian method. It is known that, for traditional convex programming, these primal-dual gradient methods have at most linear rate of convergence. In contrast, the proximal method of multipliers has an asymptotic superlinear rate of convergence \cite{Rockafellar76B}.

From Theorem \ref{th:MALM}, we observe that Algorithm \ref{alg:MALM0} posseses $\cO(\sqrt{T})$ sublinear bounds of both regret and constraint violation, i.e.,
\[
\Reg(T)\leq\cO(\sqrt{T}),\quad \Vio^{(i)}(T)\leq \cO(\sqrt{T}),\ i=1,\ldots,p,
\]
which are better than MOSP with $\cO(T^{2/3})$  bounds of both regret and constraint violation, and CL with $\cO(\sqrt{T})$  regret bound and $\cO(T^{3/4})$   constraint violation bound.

Moreover, we observe from Theorem \ref{th:MALM} that the stepsizes are set as $\alpha=\sqrt{T}$ and $\sigma=1/\sqrt{T}$, which are free of parameters,  such as the diameter $D$ of set $\cC$, the subgradient bounds $\kappa_f$ and $\kappa_g$. In contrast, the stepsizes in MOSP and CL are dependent on problem parameters, which are often difficult to estimate and hence  may cause inconvenience in implementation.

Finally, as stated by the authors in \cite{CLG2017}, MOSP is only efficient for linear constraints and computational intractable for nonlinear constraints.
\end{remark}
\subsection{Examples of Model}
Obviously, at each time slot $t$ in Algorithm \ref{alg:MALM0}, the main task is to compute $x_{t+1}$ by solving the minimization problem (\ref{eq:alg1}). If we directly use $f_t$ and $g_t$ (instead of $F_{t,x_t}$ and $G_{t,x_t}$) in  (\ref{eq:aug}) as the traditional augmented Lagrangian methods, in general, solving the corresponding  problem (\ref{eq:alg1}) is a computational challenge. However, in the following we will introduce a list of examples to show that this issue can be alleviated by choosing the model carefully.

\textit{Linearized Model.} We first introduce  the following linearized model:
$$
\begin{array}{ll}
F_{t,x_t}(x):=f_t(x_t)+\langle u_t,x-x_t\rangle,\\[10pt]
G^{(i)}_{t,x_t}(x):=g^{(i)}_{t}(x_t)+\langle v^{(i)}_t,x-x_t\rangle,\ i=1,\ldots,p,	
\end{array}
$$
where $u_t\in\partial f_t(x_t)$ and $ v^{(i)}_t\in\partial g^{(i)}_{t}(x_t)$. In this case,  the associated problem (\ref{eq:alg1}) is almost equivalent to minimizing a strongly convex quadratic function over $\cC$. Furthermore, if $p=1$, the closed-form  solution $x_{t+1}$ to (\ref{eq:alg1}) can be explicitly calculated, see an example in Section \ref{subsec:olg}.

\textit{Quadratic-linearized Model.} Suppose that $f_t(\cdot)$ is $\iota$-strongly convex, i.e., for all $x,y$ and any $u\in\partial f_t(y)$,
\[
f_t(x)\geq f_t(y)+\langle u,x-y\rangle+\frac{\iota}{2}\|x-y\|^2.
\]
Then, we can replace $F_{t,x_t}(x)$ in the linearized model with the following quadratic approximation:
$$
F_{t,x_t}(x):=f_t(x_t)+\langle u_t,x-x_t\rangle+\frac{\iota}{2}\|x-x_t\|^2.
$$
Compared with linearized model, this  is a better approximation, and the associated problem (with an extra term $\frac{\iota}{2}\|x-x^t\|^2$) is easier to be solved.

\textit{Truncated Model.} If we assume that $f_t(x)\geq 0$ for all $x$, then  we can replace $F_{t,x_t}(x)$ in the linearized model with the following truncated model:
$$
F_{t,x_t}(x):=[f_t(x_t)+\langle u_t,x-x_t\rangle]_+.
$$
Again, it is a better approximation to $f_t$ than the linearized model, and the associated minimization problem is also very simple.

\textit{Plain Model.} Finally, for certain applications, if the functions $f_t$ and $g_t$ are simple enough (such as linear functions), we can directly use the plain model:
$$
F_{t,x_t}(x):=f_t(x),\quad G^{(i)}_{t,x_t}(x):=g^{(i)}_{t}(x),\ i=1,\ldots,p.
$$

It is trivial to check that the conditions in Assumption \ref{assu:model} are satisfied for all above examples. For certain special problems, other efficient models are possible to be designed. For the general problem with subdifferentiable loss and constraint functions, the linearized model is usually applicable and efficient enough.

\section{Augmented Lagrangian Methods for Constrained OCO with Feedback Delays}\label{sec:delay}
\subsection{Problem Statement and Algorithm Development}
In this subsection, we propose a class of model-based augmented Lagrangian methods  for solving time-varying constrained OCO with delayed function feedback.

As stated in Section \ref{sec:nodelay}, in the setting of traditional constrained OCO, the feedback information of functions $f_t$ and $g_t$ is received immediately after the decision $x_t$ is submitted. However, in various practical applications  \cite{BDS2019,YLYXCJ2021,PSLW2022}, the feedback is always delayed by a few rounds. In the following, we assumed that the functions $f_t$ and $g_t$ are received by the player at time slot $t+\tau$ after the decision $x_{t+\tau}$ is submitted, where $\tau\geq 0$ is the feedback delay. In other words, at time $t\geq \tau$, after the decision $x_t$ is submitted, the player receives the delayed feedback $f_{t-\tau}$ and $g_{t-\tau}$, and then make a new decision $x_{t+1}$. The performance of a decision sequence $\{x_t\}_{t\in[T]}$ is also evaluated by $\Reg(T)$ and $\Vio^{(i)}(T)$ defined in (\ref{eq:regret}). Nevertheless, here we should assume that the given time horizon $T>\tau$, otherwise the decision sequence is meaningless since no feedback has been received yet.

We now state the  model-based augmented Lagrangian methods for OCO with delayed function feedback. Recall the  augmented Lagrangian function is defined as follows,
\be\label{eq:aug}
\cL_{t,\sigma}(x,\lambda):=F_{t,x_t}(x)+\frac{1}{2\sigma}\left[\|[\lambda+\sigma G_{t,x_t}(x)]_+\|^2-\|\lambda\|^2\right].
\ee
At time slot $t\geq\tau$, after the decision $x_t$ and the multiplier $\lambda_t$ are determined,  the player receives the delayed functions $f_{t-\tau}$ and $g_{t-\tau}$. Then, a new action $x_{t+1}$ is taken by solving the following minimization problem:
\be\label{eq:alg1}
x_{t+1}=\argmin\limits_{x\in\cC}\left\{\cL_{t-\tau,\sigma}(x,\lambda_t)+\frac{\alpha}{2}\|x-x_{t-\tau}\|^2\right\}.
\ee
 Furthermore, the player updates the multiplier $\lambda_{t+1}$ by
\be\label{eq:alg2}
\lambda_{t+1}=[\lambda_t+\sigma G_{t-\tau,x_{t-\tau}}(x_{t+1})]_+.
\ee
The detail of  the methods is summarized in Algorithm \ref{alg:MALM}.
\begin{algorithm2e}[htp]
\caption{MALM for constrained OCO with feedback delays}
\label{alg:MALM}
\lnlset{alg:MALM0-1}{1}{Initialization: ~Choose an initial action $x_0 \in \cC$ and select parameters $\sigma>0,\alpha>0$.  Set $x_1=\cdots=x_{\tau}=x_0$ and $\lambda_0=\lambda_1=\cdots\lambda_{\tau}=0\in\R^p$.}
\\ \vspace{.5ex}
\For{$t=\tau,\tau+1,\ldots,\tau+T-1$}{
\lnlset{alg:MALM0-2}{2}{Submit the decision $x_t$.
} \\ \vspace{.5ex}
\lnlset{alg:MALM0-3}{3}{Receive the delayed feedback $f_{t-\tau}$ and $g_{t-\tau}$.
} \\ \vspace{.5ex}
\lnlset{alg:MALM4}{4}{Compute the new decision $x_{t+1}$ according to (\ref{eq:alg1}).
} \\ \vspace{.5ex}
\lnlset{alg:MALM4}{5}{Update the multiplier $\lambda_{t+1}$ according to (\ref{eq:alg2}).
}
\\}
\end{algorithm2e}

\begin{remark}\label{rem:reduce}
If we take $\tau=0$, it means no delay and the OCO problem considered in this section goes back to that in Section \ref{sec:nodelay}. Moreover, Algorithm \ref{alg:MALM} reduces to Algorithm \ref{alg:MALM0}.

Since  no feedback information is revealed in the first $\tau+1$ time slots, the decision maker has to choose the first $\tau+1$ actions blindly. In Algorithm \ref{alg:MALM}, we first choose an initial action $x_0\in\cC$  and an initial multiplier $\lambda_0=0$ similarly as in Algorithm \ref{alg:MALM0}. Then, we simply let the first $\tau+1$ decisions be $x_1=\cdots=x_{\tau}=x_0$ and the first $\tau+1$ multipliers be $\lambda_1=\cdots\lambda_{\tau}=\lambda_0$.

If the projection $\Pi_{\cC}(\cdot)$ onto $\cC$ is simple to compute, which is a standard assumption in the literature of OCO, the initial decision $x_0\in\cC$ can be easily obtained. For example, we can simply let $x_0=\Pi_{\cC}(0)$.
\end{remark}
\subsection{Properties of Model}
 To proceed the analysis on the performance of Algorithm \ref{alg:MALM}, we now study the properties of  the model $F_{t,x_t}$ and $G_{t,x_t}$. For any $u_t\in\partial F_{t,x_t}(x_t)$ and $x\in\cC$, from Assumption \ref{assu:model} it follows that
\[
\begin{array}{ll}
f_t(x)\geq F_{t,x_t}(x) \geq F_{t,x_t}(x_t)+\iprod{u_t}{x-x_t}\\[3pt]
\quad=f_t(x_t)+\iprod{u_t}{x-x_t},\
\end{array}
\]
which implies $u_t\in\partial f_t(x_t)$ and hence $\partial F_{t,x_t}(x_t)\subseteq\partial f_t(x_t)$. Similarly, we have
$\partial G^{(i)}_{t,x_t}(x_t)\subseteq\partial g^{(i)}_t(x_t)$. Therefore, from Assumption \ref{assu:subgradient} we get
\be\label{eq:subgradient}
\|u_t\|\leq\kappa_f,\quad \|v_t^{(i)}\|\leq\kappa_g,\ i=1,\ldots,p,
\ee
for any $u_t\in\partial F_{t,x_t}(x_t)$ and $v^{(i)}_t\in\partial G^{(i)}_{t,x_t}(x_t)$. Moreover, for any $x\in\cC$ it holds that
\be\label{eq:a1}
F_{t,x_t}(x) \geq f_t(x_t)+\iprod{u_t}{x-x_t} \geq f_t(x_t)-\kappa_f\|x-x_t\|,
\ee
and
\be\label{eq:a2}
G^{(i)}_{t,x_t}(x)  \geq g^{(i)}_t(x_t)-\kappa_g\|x-x_t\|,\ i=1,\ldots,p.
\ee

\subsection{Performance Analysis}
In this subsection, we shall analyze the performance of Algorithm \ref{alg:MALM} and show that the decision sequence generated by  Algorithm \ref{alg:MALM} possesses both sublinear regret and sublinear constraint violation.

For simplicity, we denote $t'=t-\tau$ in the following analysis.
One of the main challenges in the analysis is to establish the bound of the multiplier $\lambda_t$. In order to derive the bound, we first present a recursive relationship between $\|\lambda_{t+1}\|^2$ and $\|\lambda_t\|^2$.
\begin{lemma}\label{lem:recur}
Under Assumptions \ref{assu:bounded}-\ref{assu:model}, for any $t\geq\tau$ we have
\[
\begin{array}{ll}
\displaystyle\frac{1}{2\sigma}[\|\lambda_{t+1}\|^2-\|\lambda_t\|^2]\\[5pt]\leq \displaystyle 2\kappa_fD+\frac{\sigma}{2}\nu_g^2+\frac{\alpha}{2}(\|\widehat{x} -x_{t'}\|^2-\|\widehat{x} -x_{t+1}\|^2)-\varepsilon_0\|\lambda_t\|.
\end{array}
\]
\end{lemma}
\begin{proof}
The optimality condition of problem (\ref{eq:alg1}) reads as follows
\be\label{eq:b1}
0\in\partial_x\cL_{t',\sigma}(x_{t+1},\lambda_t)+\alpha(x_{t+1}-x_{t'})+\cN_{\cC}(x_{t+1}),
\ee
where $\cN_{\cC}(x_{t+1})$ stands for the normal cone of $\cC$ at $x_{t+1}$. We now consider an auxiliary optimization problem
\be\label{eq:b2}
\min_{x\in\cC}\cL_{t',\sigma}(x,\lambda_t)+\frac{\alpha}{2}(\|x-x_{t'}\|^2-\|x-x_{t+1}\|^2).
\ee
Since the objective function is convex, it holds that $\tilde{x}\in\cC$ is an optimal solution to (\ref{eq:b2}) if and only if
\[
0\in\partial_x\cL_{t',\sigma}(\tilde{x},\lambda_t)+\alpha(x_{t+1}-x_{t'})+\cN_{\cC}(\tilde{x}).
\]
Therefore, in view of (\ref{eq:b1}), we obtain that $x_{t+1}$ is an optimal solution to (\ref{eq:b2}). Hence, it holds that
\be\label{eq:b3}
\begin{array}{ll}
\cL_{t',\sigma}(x_{t+1},\lambda_t)+\frac{\alpha}{2}\|x_{t+1}-x_{t'}\|^2 \\[8pt] \leq\cL_{t',\sigma}(\widehat{x} ,\lambda_t)+\frac{\alpha}{2}(\|\widehat{x} -x_{t'}\|^2-\|\widehat{x} -x_{t+1}\|^2).
\end{array}
\ee
where $\widehat{x}\in\cC$ is given in Assumption \ref{assu:slater}. From (\ref{eq:aug}), (\ref{eq:alg2}) and (\ref{eq:a1}), we have
\[
\begin{array}{ll}
\cL_{t',\sigma}(x_{t+1},\lambda_t)\\[5pt]
=F_{t',x_{t'}}(x_{t+1})+\frac{1}{2\sigma}\left[\|[\lambda_t+\sigma G_{t',x_{t'}}(x_{t+1})]_+\|^2-\|\lambda_t\|^2\right]\\[5pt]
\geq f_{t'}(x_{t'})-\kappa_f\|x_{t+1}-x_{t'}\|+\frac{1}{2\sigma}\left(\|\lambda_{t+1}\|^2-\|\lambda_t\|^2\right).
\end{array}
\]
Moreover, using (\ref{eq:aug}) again we obtain
\[
\begin{array}{ll}
\cL_{t',\sigma}(\widehat{x} ,\lambda_t)\\[5pt]
=F_{t',x_{t'}}(\widehat{x} )+\frac{1}{2\sigma}\left[\|[\lambda_t+\sigma G_{t',x_{t'}}(\widehat{x} )]_+\|^2-\|\lambda_t\|^2\right]\\[5pt]
\leq f_{t'}(\widehat{x} )+\frac{1}{2\sigma}\left(2\sigma\iprod{\lambda_t}{G_{t',x_{t'}}(\widehat{x} )}+\sigma^2\|G_{t,x_t}(\widehat{x} )\|^2\right)\\[5pt]
\leq f_{t'}(\widehat{x} )+\iprod{\lambda_t}{g_{t'}(\widehat{x} )}+\frac{\sigma}{2}\nu_g^2,
\end{array}
\]
where the facts $F_{t',x_{t'}}(\widehat{x} )\leq f_{t'}(\widehat{x} )$ and $\|[\cdot]_+\|^2\leq\|\cdot\|^2$ are used to get the first inequality, the facts $\lambda_t\geq 0$, $G_{t',x_{t'}}(\widehat{x} )\leq g_{t'}(\widehat{x} )$ and $\|G_{t',x_{t'}}(\widehat{x} )\|\leq\nu_g$ (item (iii) in Assumption \ref{assu:model}) are applied to get the second inequality. Substituting the above two results (about $\cL_{t',\sigma}(x_{t+1},\lambda_t)$ and $\cL_{t',\sigma}(\widehat{x} ,\lambda_t)$)  into (\ref{eq:b3}) and rearranging  terms, we get
\be\label{eq:b4}
\begin{array}{ll}
\frac{1}{2\sigma}\left(\|\lambda_{t+1}\|^2-\|\lambda_t\|^2\right)+\frac{\alpha}{2}\|x_{t+1}-x_{t'}\|^2\\[5pt]
\leq f_{t'}(\widehat{x} )-f_{t'}(x_{t'}) +\iprod{\lambda_t}{g_{t'}(\widehat{x} )}+\kappa_f\|x_{t+1}-x_{t'}\|+\frac{\sigma}{2}\nu_g^2\\[5pt]+\frac{\alpha}{2}(\|\widehat{x} -x_{t'}\|^2-\|\widehat{x} -x_{t+1}\|^2).
\end{array}
\ee
From Assumptions \ref{assu:bounded} and \ref{assu:subgradient} it follows that $\|x_{t+1}-x_{t'}\|\leq D$ and
\[
f_{t'}(\widehat{x} )-f_{t'}(x_{t'})\leq\kappa_f\|\widehat{x} -x_{t'}\|\leq\kappa_fD.
\]
Furthermore, from Assumption \ref{assu:slater} and  $\|\lambda_t\|\leq \sum_{i=1}^p\lambda^{(i)}_t$ it holds  that
\[
 \iprod{\lambda_t}{g_{t'}(\widehat{x} )}\leq-\varepsilon_0\sum_{i=1}^p\lambda^{(i)}_t\leq -\varepsilon_0\|\lambda_t\|.
\]
 Finally, combining these results with (\ref{eq:b4}), we derive the claim.
\end{proof}

The following simple auxiliary result will be used several times in the analysis.
\begin{lemma}\label{lem:aux}
Consider a nonnegative scalar sequence $\{w_l\}$, and suppose that there exists a positive constant $W$ such that $w_l\leq W$ for any $l$. Assume that $\tau$ and $t'=t-\tau$ are defined as before. For any $t\geq\tau$, we have
\[
\sum_{l=0}^{s-1}[w_{t'+l}-w_{t+l+1}]\leq(\tau+1) W,
\]
where $s$ is an arbitrary positive integer.
\end{lemma}
\begin{proof}
If $s\leq\tau+1$, the claim is obvious since
\[
\sum_{l=0}^{s-1}[w_{t'+l}-w_{t+l+1}]\leq \sum_{l=0}^{s-1}w_{t'+l}\leq s W\leq(\tau+1) W.
\]
If $s>\tau+1$, noting that $t'=t-\tau$ and $$\sum_{l=\tau+1}^{s-1}w_{t-\tau+l}=\sum_{l=0}^{s-\tau-2}w_{t+l+1},$$ we have
\[
\begin{array}{ll}
\displaystyle\sum_{l=0}^{s-1}[w_{t'+l}-w_{t+l+1}]\\[8pt]
\displaystyle= \left(\sum_{l=0}^{\tau}w_{t'+l}+\sum_{l=\tau+1}^{s-1}w_{t'+l}\right)-\sum_{l=0}^{s-1}w_{t+l+1}\\[8pt]
\displaystyle\leq \left(\sum_{l=0}^{\tau}w_{t'+l}+\sum_{l=\tau+1}^{s-1}w_{t-\tau+l}\right)-\sum_{l=0}^{s-\tau-2}w_{t+l+1}\\[8pt]
\displaystyle=\sum_{l=0}^{\tau}w_{t'+l}\leq(\tau+1)W.
\end{array}
\]
The proof is completed.
\end{proof}

To simplify notations, for any positive integer $s$, let us define
\be\label{eq:psi}
\psi(\sigma,\alpha,s):=\kappa_0+(\tau+1)\kappa_1\frac{\alpha}{s}+\kappa_2\sigma+\kappa_3\sigma s,
\ee
where $\kappa_0,\kappa_1,\kappa_2,\kappa_3$ are nonnegative constants given by
\be\label{eq:kappa}
\begin{array}{ll}
\kappa_0=\frac{4\kappa_f D}{\varepsilon_0},\ \kappa_1=\frac{D^2}{\varepsilon_0},\ \kappa_2=\frac{\nu_g^2}{\varepsilon_0}-\nu_g,\\[8pt] \kappa_3=2\nu_g+\frac{\varepsilon_0}{2}+\frac{8\nu_g^2}{\varepsilon_0}\log\frac{32\nu_g^2}{\varepsilon_0^2}.
\end{array}
\ee
From Assumption \ref{assu:slater} and Assumption \ref{assu:model} it follows that
\[
G_{t,x_t}^{(i)}(\widehat{x} )\leq g_{t}^{(i)}(\widehat{x} )\leq -\varepsilon_0,
\]
and hence
\[
\nu_g\geq\|G_{t,x_t}(\widehat{x} )\|\geq\|g_{t}(\widehat{x} )\|\geq \sqrt{p}\varepsilon_0\geq \varepsilon_0,
\]
which indicates that $\kappa_2\geq 0$.
Next, based on Lemma \ref{lem:recur}, we present the following uniform bound of the multiplier $\lambda_t$.
\begin{lemma}\label{lem:multiplier}
Under Assumptions \ref{assu:bounded}-\ref{assu:model}, for any $t\geq\tau$ we have
\[
\|\lambda_t\|\leq \psi(\sigma,\alpha,s),
\]
where   $s$ is an arbitrary positive integer and $\psi(\sigma,\alpha,s)$ is defined in (\ref{eq:psi}).
\end{lemma}
\begin{proof}
The claim can be derived by applying Lemma \ref{lem:l7} in the appendix, and hence the main task of this proof is to verify the conditions in Lemma \ref{lem:l7} are satisfied with respect to the sequence $\{\|\lambda_t\|\}_{t\in[T]}$. In particular, it is required to prove: (i) $|(\|\lambda_{t+1}\|-\|\lambda_t\|)|\leq\sigma\nu_g$; (ii) for any positive integer $s$, it holds that $\|\lambda_{t+s}\|-\|\lambda_t\|\leq-\frac{\sigma\varepsilon_0s}{2}$ if $\|\lambda_t\|\geq \theta$, where $\theta$ is defined by
\[
 \theta= \frac{\sigma\varepsilon_0s}{2}+\sigma \nu_g(s-1)+ \frac{\alpha(\tau+1) D^2}{\varepsilon_0s}+\frac{4\kappa_f D}{\varepsilon_0}+ \frac{\sigma \nu_g^2}{\varepsilon_0}.
 \]
In view of (\ref{eq:kappa}), $\theta$ can be rewritten as
\be\label{eq:b7}
\theta=\kappa_0+(\tau+1)\kappa_1\frac{\alpha}{s}+\kappa_2\sigma+(\nu_g+\varepsilon_0/2)\sigma s.
\ee

Noting that the operator $[\cdot]_+$ is non-expansive, that is, $\|[a]_+-[b]_+\|\leq \|a-b\|$ for all $a,b\in\R^p$, we have from (\ref{eq:alg2}) that
\[
\begin{array}{ll}
|(\|\lambda_{t+1}\|-\|\lambda_t\|)|\\[8pt]
\leq\|\lambda_{t+1}-\lambda_t\|=\|[\lambda_t+\sigma G_{t',x_{t'}}(x_{t+1})]_+-[\lambda_t]_+\|\\[8pt]
\leq \sigma\|G_{t',x_{t'}}(x_{t+1})\|\leq \sigma\nu_g,
\end{array}
\]
which gives item (i). Moreover, the above inequality also implies that $\|\lambda_{t+1}\|-\|\lambda_t\|\geq-\sigma\nu_g$, which further gives
\be\label{eq:b5}
\|\lambda_{t+l}\|-\|\lambda_t\|\geq-\sigma\nu_gl,
\ee
where $l\in\{0,1,\ldots,s-1\}$.

We now prove item (ii) under the condition $\|\lambda_t\|\geq\theta$. For any $l\in\{0,1,\ldots,s-1\}$, from Lemma \ref{lem:recur} it follows that
\[
\begin{array}{ll}
\displaystyle\frac{1}{2\sigma}[\|\lambda_{t+l+1}\|^2-\|\lambda_{t+l}\|^2]\leq 2\kappa_fD+\frac{\sigma}{2}\nu_g^2\\[5pt]
\displaystyle\quad\quad\quad+\frac{\alpha}{2}(\|\widehat{x} -x_{t'+l}\|^2-\|\widehat{x} -x_{t+l+1}\|^2)-\varepsilon_0\|\lambda_{t+l}\|,
\end{array}
\]
in which we use  $(t+l)'=t+l-\tau=t'+l$.
Making a summation over $l\in\{0,1,\ldots,s-1\}$, we have
\be\label{eq:b6}
\begin{array}{ll}
\displaystyle\frac{1}{2\sigma}[\|\lambda_{t+s}\|^2-\|\lambda_t\|^2]\\[8pt]
\displaystyle\leq(2\kappa_fD+\frac{\sigma}{2}\nu_g^2)s+\frac{\alpha(\tau+1)}{2}D^2-\varepsilon_0\sum_{l=0}^{s-1}\|\lambda_{t+l}\|,
\end{array}
\ee
where the fact that \[\sum_{l=0}^{s-1}(\|\widehat{x} -x_{t'+l}\|^2-\|\widehat{x} -x_{t+l+1}\|^2)\leq(\tau+1)D^2\] (see Lemma \ref{lem:aux}) is used.
In view of (\ref{eq:b5}) and the condition $\|\lambda_t\|\geq\theta$, we get
\[
\begin{array}{ll}
\displaystyle\sum_{l=0}^{s-1}\|\lambda_{t+l}\|\geq\sum_{l=0}^{s-1}(\|\lambda_t\|-\sigma\nu_gl)\\[5pt]
\displaystyle=s\|\lambda_t\|-\frac{s(s-1)\sigma\nu_g}{2}\geq\frac{s}{2}(\|\lambda_t\|+\theta)-\frac{s(s-1)\sigma\nu_g}{2}.
\end{array}
\]
Let us plug this inequality into (\ref{eq:b6}) and rearrange  terms,
 \[
 \|\lambda_{t+s}\|^2\leq \|\lambda_t\|^2-\sigma\varepsilon_0s\|\lambda_t\| \leq\left(\|\lambda_t\|-\frac{\sigma\varepsilon_0s}{2}\right)^2.
 \]
 This, together with the fact that $\|\lambda_t\|\geq\theta\geq\frac{\sigma\varepsilon_0s}{2}$, implies item (ii).

 Finally, observing items (i) and (ii), we have that  the conditions in Lemma \ref{lem:l7} are satisfied with respect to the sequence $\{\|\lambda_t\|\}_{t\in[T]}$. Therefore, it follows that
 \[
 \begin{array}{ll}
 \|\lambda_t\|&\leq \theta + s\sigma\nu_g+s\frac{8\sigma\nu_g^2}{\varepsilon_0}\log\frac{32\nu_g^2}{\varepsilon_0^2}\\[5pt]
 &\overset{(\mbox{a})}{=}\kappa_0+(\tau+1)\kappa_1\frac{\alpha}{s}+\kappa_2\sigma+\kappa_3\sigma s=\psi(\sigma,\alpha,s),
 \end{array}
 \]
 where (a) is due to (\ref{eq:b7}) and (\ref{eq:kappa}).
 The proof is completed.
\end{proof}

Now, we are ready to present the sublinear regret of  Algorithm \ref{alg:MALM} for constrained OCO if we set the algorithm parameters  as $\alpha=\sqrt{T/(\tau+1)}$ and $\sigma=\sqrt{(\tau+1)/T}$.
\begin{theorem}\label{th:regret}
Suppose Assumptions \ref{assu:bounded}-\ref{assu:model} hold. Set $\alpha=\sqrt{\frac{T}{\tau+1}}$ and $\sigma=\sqrt{\frac{\tau+1}{T}}$. Then, the regret of Algorithm \ref{alg:MALM} satisfies
\[
\sum_{t=0}^{T-1}[f_t(x_t)-f_t(x^*)]\leq\frac{\kappa_f^2+\nu_g^2+D^2}{2}\sqrt{(\tau+1) T},
\]
where $x^*$ is defined in (\ref{eq:xstar}).
\end{theorem}
\begin{proof}
For any $t\geq\tau$, let us replace $\widehat{x}$ with $x^*$ in (\ref{eq:b4}) and rearrange terms,
\[
\begin{array}{ll}
f_{t'}(x_{t'})-f_{t'}(x^* )+\frac{1}{2\sigma}\left(\|\lambda_{t+1}\|^2-\|\lambda_t\|^2\right)\\[5pt]
\leq  -\frac{\alpha}{2}\|x_{t+1}-x_{t'}\|^2+\kappa_f\|x_{t+1}-x_{t'}\|+\iprod{\lambda_t}{g_{t'}(x^* )}+\frac{\sigma}{2}\nu_g^2\\[5pt]+\frac{\alpha}{2}(\|x^* -x_{t'}\|^2-\|x^* -x_{t+1}\|^2).
\end{array}
\]
Combining with the facts that \[-\frac{\alpha}{2}\|x_{t+1}-x_{t'}\|^2+\kappa_f\|x_{t+1}-x_{t'}\|\leq\frac{\kappa_f^2}{2\alpha}\] and $\iprod{\lambda_t}{g_{t'}(x^* )}\leq 0$ (due to $\lambda_t\geq 0$ and $g_{t'}(x^*)<0$), we obtain
\[
\begin{array}{ll}
f_{t'}(x_{t'})-f_{t'}(x^* )+\frac{1}{2\sigma}\left(\|\lambda_{t+1}\|^2-\|\lambda_t\|^2\right)\\[5pt]
\leq  \frac{\kappa_f^2}{2\alpha}+\frac{\sigma}{2}\nu_g^2+\frac{\alpha}{2}(\|x^* -x_{t'}\|^2-\|x^* -x_{t+1}\|^2).
\end{array}
\]
Summing it for $t$ running from $\tau$ to $\tau+T-1$,
we get
\be\label{eq:b8}
\sum_{t=\tau}^{\tau+T-1}[f_{t'}(x_{t'})-f_{t'}(x^* )]\leq \frac{\kappa_f^2}{2\alpha}T+\frac{\sigma \nu_g^2}{2}T+\frac{\alpha}{2}(\tau+1)D^2,
\ee
where we have used $\lambda_{\tau}=0$ and the fact (take $t=\tau$ in Lemma \ref{lem:aux}) that
\[
\begin{array}{ll}
\displaystyle\sum_{t=\tau}^{\tau+T-1}(\|x^* -x_{t'}\|^2-\|x^* -x_{t+1}\|^2)\\[12pt]
\displaystyle=\sum_{l=0}^{T-1}(\|x^* -x_{l}\|^2-\|x^* -x_{\tau+l+1}\|^2)\leq(\tau+1)D^2.
\end{array}
\]
Finally, using $\alpha=\sqrt{T/(\tau+1)}, \sigma=\sqrt{(\tau+1)/T}$ and that fact that
\[
\sum_{t=0}^{T-1}[f_t(x_t)-f_t(x^*)]=\sum_{t=\tau}^{\tau+T-1}[f_{t'}(x_{t'})-f_{t'}(x^* )],
\]
we derive the claim from (\ref{eq:b8}).
\end{proof}

From Theorem \ref{th:regret}, it follows that the regret of Algorithm \ref{alg:MALM} for constrained OCO with feedback delays is bounded by $\Reg(T)\leq\cO(\sqrt{\tau T})$, which is in the same order as \cite{CZP2021}. However, in contrast to the parameter-free stepsizes in Algorithm \ref{alg:MALM}, the stepsize in the Algoirthm of \cite{CZP2021} is dependent on the subgradient bound of the constraint functions.

In the following theorem, it is shown that Algorithm \ref{alg:MALM} posseses sublinear constraint violation if we set the algorithm parameters  as $\alpha=\sqrt{T/(\tau+1)}, \sigma=\sqrt{(\tau+1)/T}$ and the time horizon $T$ is large enough such that $T>(\tau+1)p\kappa_g^2$.
\begin{theorem}\label{th:constraint}
Suppose Assumptions \ref{assu:bounded}-\ref{assu:model} hold. Set $\alpha=\sqrt{\frac{T}{\tau+1}}$ and $\sigma=\sqrt{\frac{\tau+1}{T}}$. Assume $T$ is large enough such that $T>(\tau+1)p\kappa_g^2$.  Then, for $i=1,\ldots,p$, the constraint violation of Algorithm \ref{alg:MALM} satisfies
\[
\begin{array}{ll}
\displaystyle\sum_{t=0}^{T-1}g^{(i)}_t(x_t)\\[10pt]
\leq[\kappa_2+2\sqrt{p}\kappa_g^2(\nu_g+\kappa_2)(\tau+1)]
+(\kappa_0+\kappa_1)(\tau+1)^{-\frac{1}{2}}\sqrt{T}\\[8pt]
+[2\kappa_3+2\kappa_g\kappa_f+2\sqrt{p}\kappa_g^2(\kappa_0+\kappa_1)](\tau+1)^{\frac{1}{2}}\sqrt{T}\\[8pt]
+4\sqrt{p}\kappa_g^2\kappa_3(\tau+1)^{\frac{3}{2}}\sqrt{T},
\end{array}
\]
where $\kappa_0,\kappa_1,\kappa_2,\kappa_3$ are constants given in (\ref{eq:kappa}).
\end{theorem}
\begin{proof}
Consider $t\geq\tau$.
For any $i=1,\ldots,p$, let us first denote
\be\label{eq:c3}
a_i=[\lambda^{(i)}_t+\sigma G^{(i)}_{t',x_{t'}}(x_{t'})]_+,\quad b_i=[\lambda^{(i)}_t+\sigma G^{(i)}_{t',x_{t'}}(x_{t+1})]_+.
\ee
Noting that $G^{(i)}_{t',x_{t'}}(x_{t'})=g_{t'}^{(i)}(x_{t'})$, we have from (\ref{eq:a2}) that
\be\label{eq:c1}
a_i-b_i\leq \sigma|G^{(i)}_{t',x_{t'}}(x_{t'})- G^{(i)}_{t',x_{t'}}(x_{t+1})|\leq \sigma\kappa_g\|x_{t+1}-x_{t'}\|.
\ee
In view of the fact that $[d+e]_+\leq|d|+|e|$, it follows that
\be\label{eq:c2}
\begin{array}{ll}
\displaystyle\sum_{i=1}^pb_i\leq \sum_{i=1}^p[|\lambda^{(i)}_t|+\sigma |G^{(i)}_{t',x_{t'}}(x_{t+1})|]\\[8pt]
\leq\sqrt{p}[\|\lambda_t\|+\sigma\|G_{t',x_{t'}}(x_{t+1})\|]\leq \sqrt{p}[\|\lambda_t\|+\sigma\nu_g].
\end{array}
\ee
From (\ref{eq:c1}), (\ref{eq:c2}) and the fact that $$a_i^2-b_i^2=(a_i-b_i)^2+2b_i(a_i-b_i),$$ it holds that
\be\label{eq:c4}
\begin{array}{ll}
\displaystyle\sum_{i=1}^p(a_i^2-b_i^2)=\sum_{i=1}^p[(a_i-b_i)^2+2b_i(a_i-b_i)]\\[6pt]
\displaystyle\leq p\sigma^2\kappa^2_g\|x_{t+1}-x_{t'}\|^2+2\sigma\kappa_g\|x_{t+1}-x_{t'}\|\sum_{i=1}^pb_i\\[8pt]
\leq p\sigma^2\kappa^2_g\|x_{t+1}-x_{t'}\|^2+2\sqrt{p}(\|\lambda_t\|+\sigma\nu_g)\sigma\kappa_g\|x_{t+1}-x_{t'}\|.
\end{array}
\ee

We now try to get a bound of $\|x_{t+1}-x_{t'}\|$.
Replacing $\widehat{x}$ with $x_{t'}$ in (\ref{eq:b3}), we have
\[
\begin{array}{ll}
\cL_{t',\sigma}(x_{t+1},\lambda_t)+\frac{\alpha}{2}\|x_{t+1}-x_{t'}\|^2\\[8pt]  \leq\cL_{t',\sigma}(x_{t'} ,\lambda_t)-\frac{\alpha}{2}\|x_{t'} -x_{t+1}\|^2.
\end{array}
\]
By using (\ref{eq:c3}) and the definition of the augmented Lagrangian function  in (\ref{eq:aug}), the above inequality can be rewritten as
\[
\alpha\|x_{t+1}-x_{t'}\|^2\leq f_{t'}(x_{t'})-F_{t,x_{t'}}(x_{t+1})+\frac{1}{2\sigma}\sum_{i=1}^p(a_i^2-b_i^2).\]
Then, it follows from  (\ref{eq:a1}) and (\ref{eq:c4}) that
\[
\begin{array}{ll}
\alpha\|x_{t+1}-x_{t'}\|^2\leq \kappa_f\|x_{t+1}-x_{t'}\|+\frac{p\kappa^2_g\sigma}{2}\|x_{t+1}-x_{t'}\|^2\\[8pt]\quad\quad+\sqrt{p}(\|\lambda_t\|+\sigma\nu_g)\kappa_g\|x_{t+1}-x_{t'}\|.
\end{array}
\]
 Dividing both sides by $\|x_{t+1}-x_{t'}\|$ and rearranging terms, we get the bound of $\|x_{t+1}-x_{t'}\|$ as follows,
\[
\|x_{t+1} -x_{t'}\|\leq\frac{2}{2\alpha-p\kappa_g^2\sigma}(\kappa_f+\sqrt{p}\kappa_g\|\lambda_t\|+\sqrt{p}\kappa_g\nu_g\sigma).
\]
Here, we have used the fact that $2\alpha-p\kappa_g^2\sigma>0$  which is from the conditions that $\alpha=\sqrt{T/(\tau+1)}, \sigma=\sqrt{(\tau+1)/T}$ and $T>(\tau+1)p\kappa_g^2$. In fact, it also holds that $2\alpha-p\kappa_g^2\sigma>\alpha$ and hence
\be\label{eq:c6}
\|x_{t+1} -x_{t'}\|\leq \frac{2}{\alpha}(\kappa_f+\sqrt{p}\kappa_g\|\lambda_t\|+\sqrt{p}\kappa_g\nu_g\sigma).
\ee
Recall that $\|\lambda_t\|\leq \psi(\sigma,\alpha,s)$ in Lemma \ref{lem:multiplier}, from (\ref{eq:psi}) we have
\be\label{eq:c7}
\begin{array}{ll}
\|\lambda_t\|\leq\kappa_0+(\tau+1)\kappa_1\frac{\alpha}{s}+\kappa_2\sigma+\kappa_3\sigma s\\[8pt]
\leq\kappa_0+\kappa_1+\kappa_2\sigma+2\kappa_3(\tau+1),
\end{array}
\ee
where we set $s$ to be an integer such that $\alpha/s\leq 1/(\tau+1)$ and $\sigma s\leq 2(\tau+1)$, i.e., $\sqrt{T(\tau+1)}\leq s\leq 2\sqrt{T(\tau+1)}$.

Next, we aim to analyze the constraint violation. It follows from (\ref{eq:alg2}) and (\ref{eq:a2}) that
\[
\begin{array}{ll}
\lambda_{t+1}^{(i)}=[\lambda^{(i)}_t+\sigma G_{t',x_{t'}}^{(i)}(x_{t+1})]_+\geq \lambda^{(i)}_t+\sigma G_{t',x_{t'}}^{(i)}(x_{t+1})\\[8pt]
\geq \lambda^{(i)}_t+\sigma [g_{t'}^{(i)}(x_{t'})-\kappa_g\|x_{t+1}-x_{t'}\|],
\end{array}
\]
which can be rewritten as
\[
g_{t'}^{(i)}(x_{t'})\leq\frac{1}{\sigma}(\lambda_{t+1}^{(i)}-\lambda_{t}^{(i)})+\kappa_g\|x_{t+1}-x_{t'}\|.
\]
Summing it for $t$ running from $\tau$ to $\tau+T-1$ and using $\lambda_{\tau}=0$, we obtain
\[
\sum_{t=\tau}^{\tau+T-1}g_{t'}^{(i)}(x_{t'})\leq \frac{\lambda_{\tau+T}^{(i)}}{\sigma}+\kappa_g\sum_{t=\tau}^{\tau+T-1}\|x_{t+1}-x_{t'}\|.
\]
Combining with $\lambda_{\tau+T}^{(i)}\leq\|\lambda_{\tau+T}\|$ and (\ref{eq:c6}), we have
\[
\begin{array}{ll}
\displaystyle\sum_{t=0}^{T-1}g^{(i)}_t(x_t)=\sum_{t=\tau}^{\tau+T-1}g_{t'}^{(i)}(x_{t'})\\[10pt]
\displaystyle\leq \frac{\|\lambda_{\tau+T}\|}{\sigma}+\frac{2\kappa_gT(\kappa_f+\sqrt{p}\kappa_g\|\lambda_t\|+\sqrt{p}\kappa_g\nu_g\sigma)}{\alpha}.
\end{array}
\]
Finally, using (\ref{eq:c7}), we get the claim.
\end{proof}

From Theorem \ref{th:constraint}, it follows that the constraint violation of Algorithm \ref{alg:MALM} for constrained OCO with feedback delays is bounded by
\[
\Vio^{(i)}(T)\leq\cO\left(T^{\frac{1}{2}}\tau^{\frac{3}{2}}\right).
\]
In contrast, in \cite{CZP2021} the constraint violation is bounded by
\[
\Vio^{(i)}(T)\leq\cO\left(T^{\frac{3}{4}}\tau^{\frac{1}{4}}\right).
\]
Apparently, when the delay is far less than the time horizon, i.e., $\tau\ll T$, our result is better than \cite{CZP2021}. Moreover, the assumption that $\lim_{T\rightarrow\infty}\frac{\tau(T)}{T}=0$ given in \cite{CZP2021} is not required in this paper. We also admit that the bound $\cO\left(T^{\frac{3}{4}}\tau^{\frac{1}{4}}\right)$ is better than $\cO\left(T^{\frac{1}{2}}\tau^{\frac{3}{2}}\right)$ when the delay $\tau$ is large enough such that $\tau\geq T^{1/5}$.

Finally, let us point out that during the whole discussion in  this section we are able to take $\tau=0$. In this case, it reduces to the setting of Section \ref{sec:nodelay}, and Theorem \ref{th:MALM} is derived by Theorem \ref{th:regret} and Theorem \ref{th:constraint}. However, it is of its own value to keep Section \ref{sec:nodelay} in the current form, since constrained OCO without delay is more common and popular in practice than delayed OCO.

\section{Numerical Experiments}\label{sec:numerical}
In this section, numerical results of the proposed algorithms are presented. We study three numerical examples: online network resource allocation, online logistic regression and online quadratically constrained quadratical program. The first two are examples of constrained OCO without feedback delay for testing Algorithm \ref{alg:MALM0}. The third one is an example of constrained OCO with delayed feedback for testing Algorithm \ref{alg:MALM}. Compared with several state-of-the-art algorithms, the proposed algorithms are demonstrated to be effective.
\subsection{Online Network Resource Allocation}\label{subsec:nra}
In this subsection, we consider the following online network resource allocation problem. Given a cloud network represented by a directed graph $\cG=(\cI,\cE)$ with node set $\cI$ and edge set $\cE$, where $|\cI|=I$ and $|\cE|=E$. The node  set $\cI=\cJ\bigcup\cK$ contains mapping nodes collected in the set $\cJ=\{1,\ldots,J\}$ and data centers collected in the set $\cK=\{1,\ldots,K\}$. At time $t$, each mapping node $j$ receives an exogenous data request $b_t^j$ and forwards the amount $z_t^{jk}$ to each data center $k$. Each data center $k$ schedules workload $y_t^k$, which can be viewed as the weight of a virtual outgoing edge $(k,\ast)$. So, the edge set $\cE=\{(j,k),\forall j\in\cJ, k\in\cK\}\bigcup\{(k,\ast),\forall k\in\cK\}$ includes all the edges from mapping nodes to data centers, and all virtual outgoing edges of the data centers.

The online network resource allocation problem can be regarded as an example of constrained OCO (without feedback delay) with
\[
\begin{array}{ll}
\displaystyle f_t(x_t):=\sum_{j\in\cJ}\sum_{k\in\cK}c^{jk}(z_t^{jk})^2+\sum_{k\in\cK}p_t^k(y_t^k)^2,\\[15pt]
g_t(x_t):=Ax_t+b_t,
\end{array}
\]
where, $x_t:=[z_t^{11},z_t^{21},\ldots,z_t^{JK},y_t^1,\ldots,y_t^K]^T\in\R^{E}$ denotes the resource allocation vector, $c^{jk}$ is the bandwidth cost for transmitting from mapping node $j$ to data center $k$, $p_t^k$ is the energy price at data center $k$, $b_t:=[b_t^1,\ldots,b_t^J,0,\ldots,0]^T\in\R^{I}$, and the $I\times E$ node-incidence matrix $A$ is formed with the $(i,e)$-th entry
\[
A(i,e)=\left\{
\begin{array}{ll}
1,\quad &\mbox{if edge } e \mbox{ enters node } i,\\[3pt]
-1,\quad &\mbox{if edge } e \mbox{ leaves node } i,\\[3pt]
0,\quad &\mbox{else}.
\end{array}
\right.
\]
The constraint $Ax_t+b_t\leq 0$ represents the workload flow conservation.
In addition, the set $\cC:=\{0\leq x\leq\bar{x}\}$, where $\bar{x}:=[\bar{z}_t^{11},\bar{z}_t^{21},\ldots,\bar{z}_t^{JK},\bar{y}_t^1,\ldots,\bar{y}_t^K]^T$ contains the bandwidth limit $\bar{z}_t^{jk}$ of edge $(j,k)$ and the resource capability $\bar{y}_t^k$ of data center $k$.

We shall test the following numerical case provided in \cite{CLG2017}. Consider the allocation task with   $J=10$ and $K=10$. The bandwidth limit $\bar{z}_t^{jk}$ is uniformly randomly generated within $[10,100]$, and the resource capability $\bar{y}_t^k$ is uniformly randomly generated from $[100,200]$. The bandwidth cost is set as $c^{jk}=40/\bar{z}_t^{jk}$. The energy price is set as $p_t^k=\sin(\pi t/12)+n_t^k$ with noise $n_t^k$ uniformly distributed over $[1,3]$, and the data request is set as $b_t^j=50\sin(\pi t/12)+v_t^j$ with noise $v_t^j$ uniformly distributed over $[99,101]$. Finally, we set the time horizon $T=10000$.

We apply Algorithm \ref{alg:MALM0} (denoted by MALM) for solving this numerical problem, in which we use \textit{Plain Model} to form (\ref{eq:approx}) and the associated algorithm. The stepsizes in Algorithm \ref{alg:MALM0} are set to $\alpha=0.1\sqrt{T}$ and $\sigma=100/\sqrt{T}$. The initial decision is set to $x_0=0$. The performance is measured by the time average regrets $\frac{\Reg(t)}{t}$ and the time average constraint violation $\max_i\frac{\Vio^{(i)}(t)}{t}$, where $\Reg(t)$ and $\Vio^{(i)}(t)$ are defined in (\ref{eq:regret}), and the best fixed decision $x^*$ is computed by the MATLAB built-in function \texttt{quadprog} for solving the corresponding problem (\ref{eq:xstar}).

In this experiment, we intend to compare the performance of MALM with the following well-studied algorithms:
\begin{itemize}
\item MOSP. The modified online saddle-point (MOSP) algorithm proposed in \cite{CLG2017} is in the following form:
\[\left\{
\begin{array}{ll}
\displaystyle x_{t+1}=\argmin_{x\in\cC}\nabla f_t(x_t)^Tx+\lambda_t^Tg_t(x)+\frac{\|x-x_t\|^2}{2\alpha},\\[10pt]
\lambda_{t+1}=[\lambda_t+\mu g_t(x_{t+1})]_+.
\end{array}
\right.
\]
Since $g_t$ is linear in this experiment, we have the following equivalent reformulation of MOSP:
\[\left\{
\begin{array}{ll}
\displaystyle x_{t+1}=\Pi_{\cC}[x_t-\alpha(\nabla f_t(x_t)+J_t(x_t)\lambda_t)],\\[10pt]
\lambda_{t+1}=[\lambda_t+\mu g_t(x_{t+1})]_+,
\end{array}
\right.
\]
where $J_t(x_t)=A^T$ is the Jacobian matrix of $g_t$ at $x_t$.
We set the parameters to $\alpha=\mu=T^{-1/3}$.
\item CL. The algorithm (simply denoted by CL) proposed in \cite{CLiu2019} is formed as
\[\left\{
\begin{array}{ll}
\displaystyle x_{t+1}=\Pi_{\cC}[x_t-\eta(\nabla f_t(x_t)+J_t(x_t)\lambda_t)],\\[10pt]
\lambda_{t+1}=[\lambda_t+\eta(g_t(x_{t})-\delta\eta\lambda_t)]_+.
\end{array}
\right.
\]
We set the parameters to $\eta=2T^{-1/2}$ and $\delta=0.01$.
\item NY. The algorithm (simply denoted by NY) proposed in \cite{NYu2017} can be equivalently written as
\[\left\{
\begin{array}{ll}
\displaystyle x_{t+1}=\Pi_{\cC}\left[x_t-\frac{\nu\nabla f_t(x_t)+J_t(x_t)\lambda_t}{2\alpha}\right],\\[10pt]
\lambda_{t+1}=[\lambda_t+g_t(x_{t+1})]_+.
\end{array}
\right.
\]
We set the parameters to $\alpha=T$ and $\nu=T^{1/2}$.

\end{itemize}

We can observe that all of these three algorithms are of very similar formulations. All of them belong to the primal-dual gradient method associated with a type of Lagrangian function. For example, in CL the new action $x_{t+1}$ can be rewritten as a primal gradient descent step as
\[
x_{t+1}=\Pi_{\cC}(x_t-\eta\nabla_x\cL(x_t,\lambda_t)),
\]
and $\lambda_{t+1}$ can be performed as a dual gradient ascent step as
\[
\lambda_{t+1} =\Pi_{\R_+^p}(\lambda_t+\eta\nabla_{\lambda}\cL(x_t,\lambda_t)),
\]
where $\R_+^p:=\{x\in\R^p:x\geq 0\}$ and $\cL$ is a modified Lagrangian defined by
\[
\cL(x,\lambda) = f_t(x)+\lambda^Tg_t(x)-\frac{\delta}{2}\|\lambda\|^2.
\]

The numerical results are shown in Fig. \ref{fig:nra}, in which we can see the performances of the four tested algorithms. We observe that, MOSP, CL and YN perform very well for this online network resource allocation problem, and the performance of MALM is at least comparable. Let us point out that, this online network resource allocation problem is actually very simple (with quadratical loss $f_t$ and linear constraint $g_t$), thus the advantage of MALM is not obvious compared with these primal-dual gradient methods. In addition, it is shown in Fig. \ref{nra-regret} that the time average regrets all converges to negative values, i.e., $\frac{1}{T}\sum_{t=0}^{T-1}f_t(x_t)<\frac{1}{T}\sum_{t=0}^{T-1}f_t(x^*)$. This is possible in theory, since $x^*$ is only the fixed best decision, not the dynamic best decision $x_t^*$ defined in (\ref{eq:dynamic-decision}).

\begin{figure}[!ht]
\centering
\subfloat[]{\includegraphics[width=2.9in]{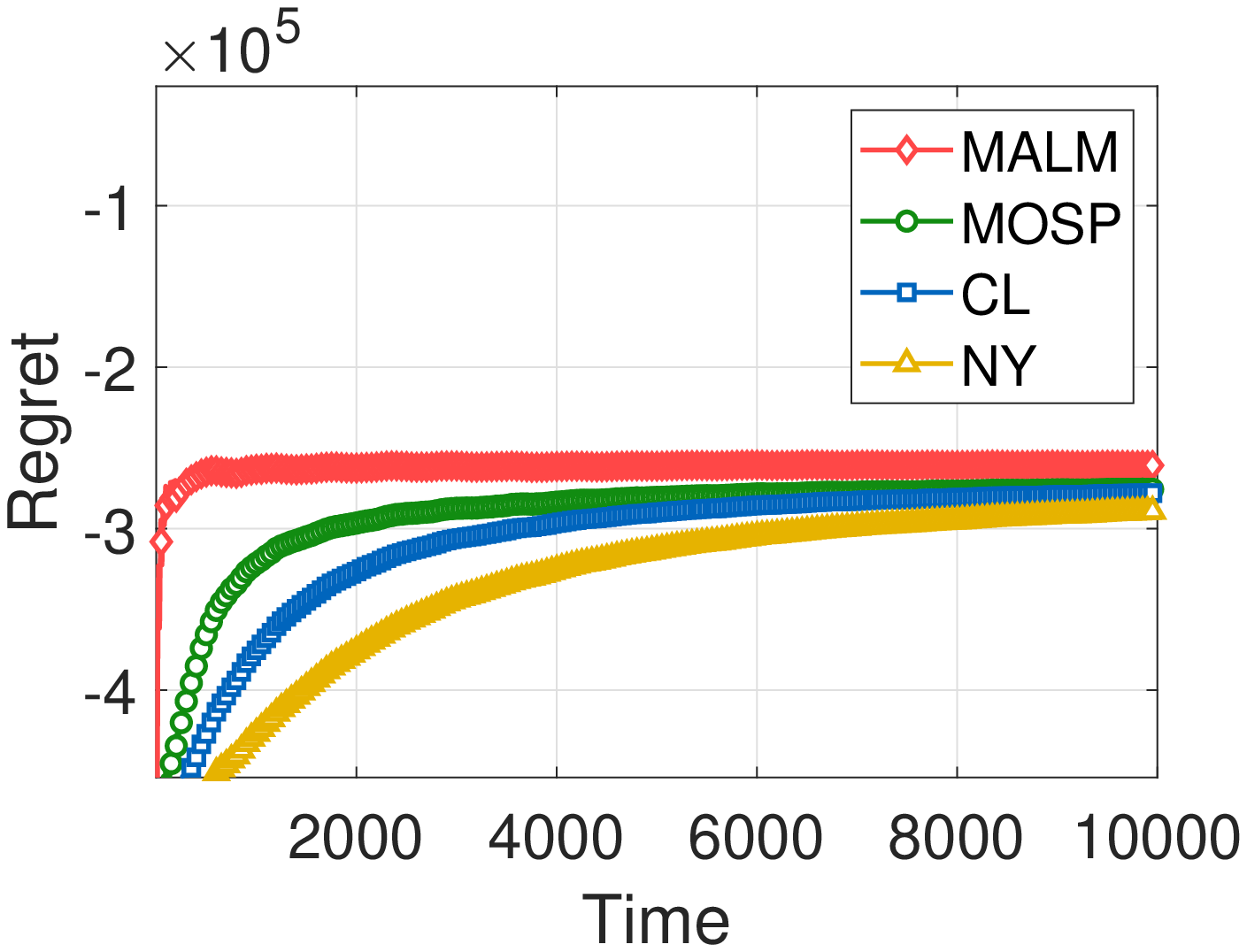}%
\label{nra-regret}}
\\
\subfloat[]{\includegraphics[width=2.9in]{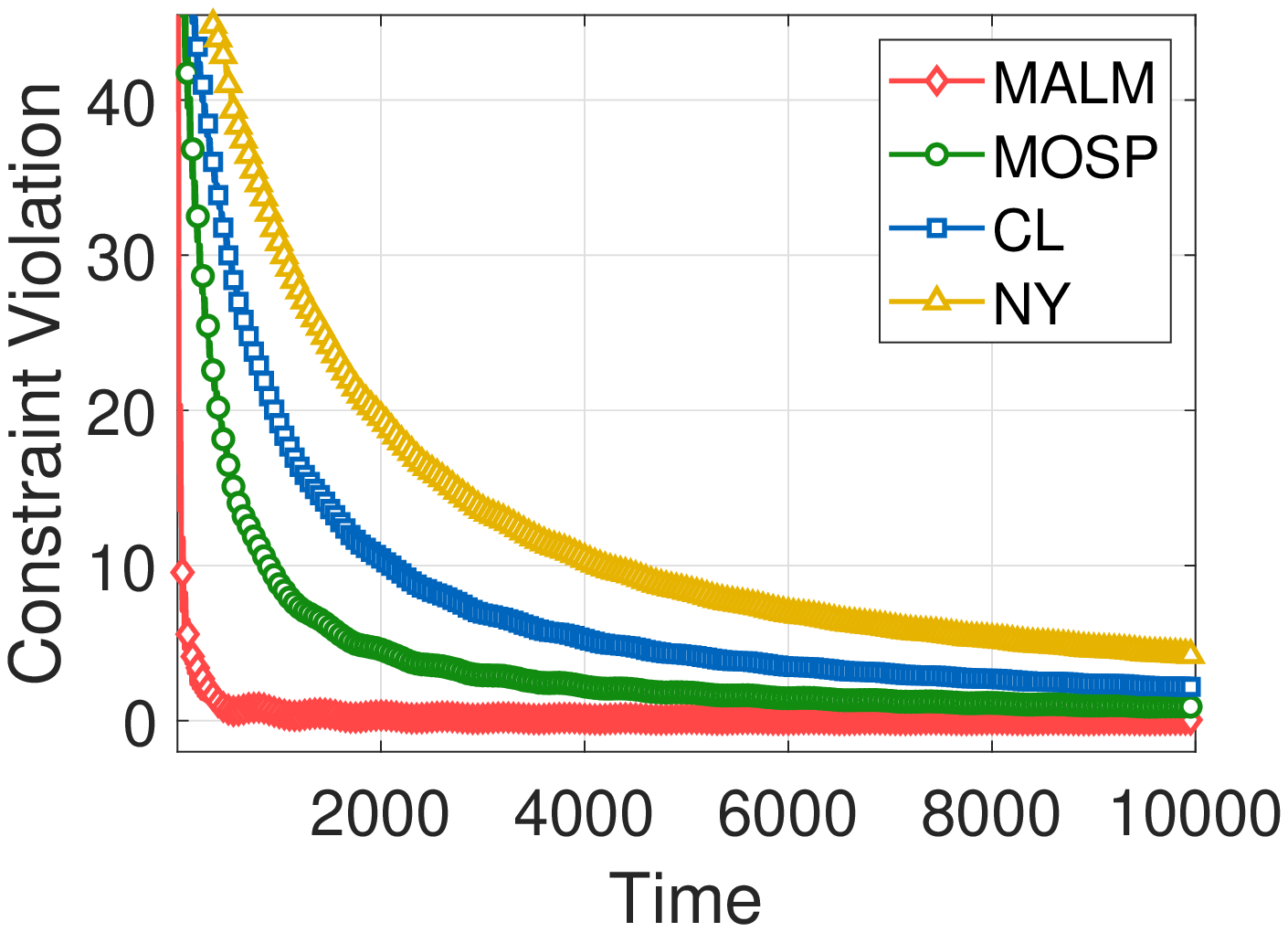}%
\label{nra-vio}}
\caption{Comparison of algorithms with respect to time average regrets and time average constraint violation for online network resource allocation. (a) Time average regrets. (b) Time average constraint violations.}
\label{fig:nra}
\end{figure}

\subsection{Online Logistic Regression}\label{subsec:olg}
In this subsection, we study the following online logistic regression problem, which is an example of constrained OCO
(without feedback delay) with
\[
\begin{array}{ll}
\displaystyle f_t(x):=\sum_{i=1}^k\log(1+\exp(-l_{i,t}u_{i,t}^Tx)),\\[15pt]
g_t(x):=\|x\|_1-a_t,
\end{array}
\]
where $u_{i,t}\in\R^n$ is the $i$-th training data at time $t$, $l_{i,t}\in\{-1,1\}$ is the corresponding label, and $a_t$ is a threshold on $\|x\|_1$ to guarantee sparsity. The set $\cC:=\{x\in\R^n:\|x\|_{\infty}\leq M\}$.

We will test the following numerical case provided in \cite{CLiu2019}. The data is generated recursively as $u_{i,t+1}=u_{i,t}+\beta_{i,t}$, where $\beta_{i,t}\in\R^n$ is uniformly distributed over $\left[-\frac{1}{2t},\frac{1}{2t}\right]$. The label $l_{i,t}$ is uniformly randomly chosen from $\{-1,1\}$. The threshold is generated as $a_{t+1}=[a_t+\zeta_t]_+$, where $\zeta_{t}\in\R$ is uniformly distributed over $\left[-\frac{1}{2t},\frac{1}{2t}\right]$. We also set $n=5$, $k=10$, $T=5000$.

The performance of the proposed algorithm is also measured by the time average regrets $\frac{\Reg(t)}{t}$ and the time average constraint violation $\frac{\Vio(t)}{t}$, where $\Reg(t)$ and $\Vio(t)$ are defined in (\ref{eq:regret}), and the best fixed decision $x^*$ is computed by CVX\cite{cvx} for solving the corresponding problem (\ref{eq:xstar}).

We apply Algorithm \ref{alg:MALM0} (also denoted by MALM) for solving this numerical problem, in which we use \textit{Linearized Model}. In this setting, the corresponding subproblem (\ref{eq:MALM-1}) is rewritten as
\[
x_{t+1}=\argmin_{x\in\cC}\left\{\frac{\alpha}{2}\|x\|^2+a_t^Tx+\frac{\sigma}{2}[b_t^Tx+\gamma_t]_+^2\right\},
\]
where $\alpha,\sigma$ are stepsizes given in Algorithm \ref{alg:MALM0},  $a_t, b_t, \gamma_t$ are denoted by
\[
a_t:=-\alpha x_t+u_t,\quad b_t:=v_t,\quad \gamma_t:=\frac{\lambda_t}{\sigma}+g_t(x_t)-v_t^Tx_t.
\]
Here, $u_t$ and $v_t$ are (sub)gradients of $f_t$ and $g_t$ at $x_t$, respectively. By a simple calculation, we derive the closed-form solution $x_{t+1}=\Pi_{\cC}(\bar{x}_t)$, where $\bar{x}_t$ is computed by
\[
\bar{x}_t=\left\{
\begin{array}{ll}
-\frac{a_t}{\alpha},\quad &\mbox{if }\alpha\gamma_t\leq a_t^Tb_t,\\[8pt]
-\frac{1}{\alpha}(a_t+\frac{\sigma\gamma_t}{2}b_t)+\frac{\sigma b_tb_t^T}{2\alpha^2+\alpha\sigma b^Tb},\quad &\mbox{if }\alpha\gamma_t> a_t^Tb_t.
\end{array}
\right.
\]
We also set $\alpha=10\sqrt{T}$, $\sigma=10/\sqrt{T}$.

In this experiment, we  compare the performance of MALM with CL and NY, which are  described previously. The time average regrets and time average constraint violations are shown in Fig. \ref{olr-regret} and Fig. \ref{olr-vio}, respectively. In both two figures, we can observe that, the performances of CL and NY are very close, and MALM is apparently superior.

\begin{figure}[!ht]
\centering
\subfloat[]{\includegraphics[width=2.9in]{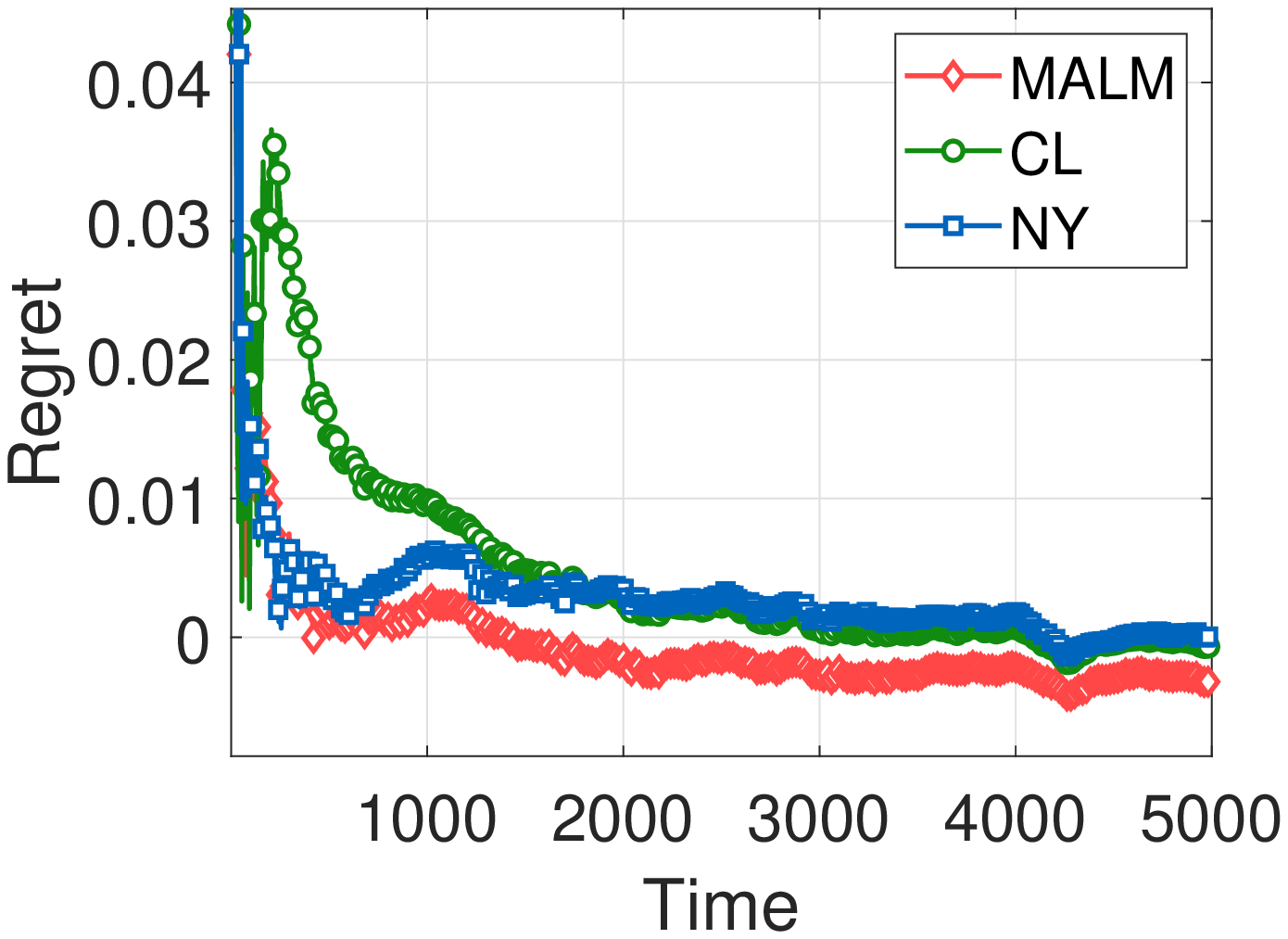}%
\label{olr-regret}}
\\
\subfloat[]{\includegraphics[width=2.9in]{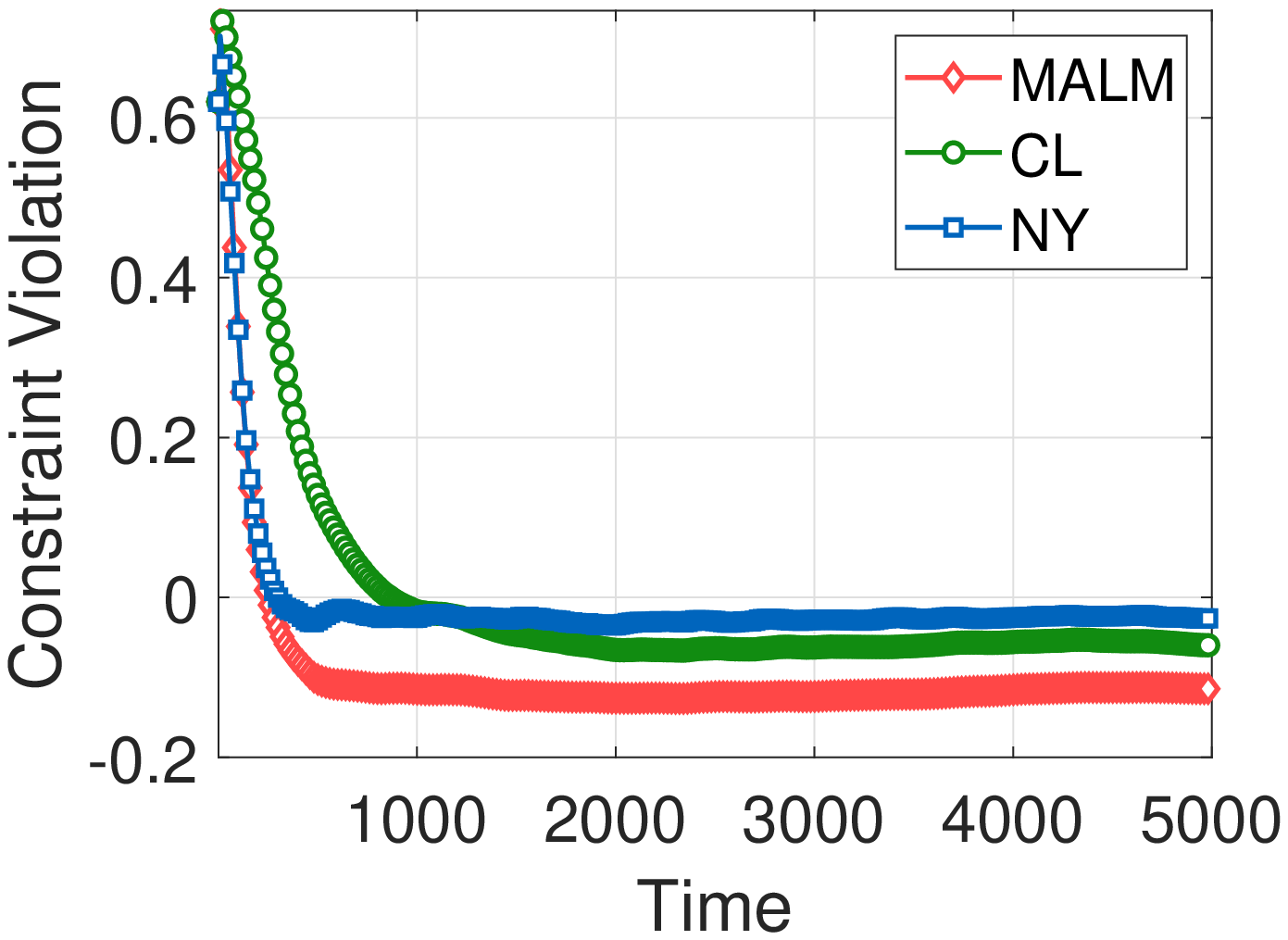}%
\label{olr-vio}}
\caption{Comparison of algorithms with respect to time average regrets and time average constraint violation for online logistic regression. (a) Time average regrets. (b) Time average constraint violations.}
\label{fig:olr}
\end{figure}


\subsection{Online Quadratically Constrained Quadratical Program}
In this subsection, we consider the online quadratically constrained quadratical program (OQCQP), which is also an instance of constrained OCO  with
\[
\begin{array}{ll}
\displaystyle f_t(x_t):=\frac{1}{2}x^TA_tx+b_t^Tx,\\[10pt]
\displaystyle g_t^{(i)}(x):=\frac{1}{2}x^TC_t^{(i)}x+(d_t^{(i)})^Tx+e_t^{(i)},\ i=1,\ldots,p,
\end{array}
\]
where $A_t\in\cS_{+}^n$, $b_t\in\R^n$, $C_t^{(i)}\in\cS_{+}^n$, $d_t^{(i)}\in\R^n$ and $e_t^{(i)}\in\R$. Here, $\cS_{+}^n$ denotes the set of all $n\times n$ positive definite matrices. The set $\cC:=\{x\in\R^n:\|x\|\leq R\}$. OQCQP arises in various applications, such as signal processing,  power generation \cite{YLYXCJ2021}, optimal power flow \cite{BDS2019}, etc. Due to communication delays or environment reaction time, the data or user feedback are usually revealed to the player with delays. In this experiment, we assume that $\tau\geq 0$ is the feedback delay and the information of $\{A_t, b_t, C_t^{(i)},d_t^{(i)},e_t^{(i)}\}$ is available at time slot $t+\tau$ after the decision $x_{t+\tau}$ is submitted. Apparently, $\tau=0$ means no delay.

The following test case of OQCQP is constructed in \cite{CZP2021}. Let $A_1$ be the identity matrix. We generate $A_{t+1}$ as follows: let $\tilde{A}_t=A_t+\Delta_t$ where $\Delta_t$ is a symmetric matrix and its entry is uniformly distributed over $[-0.1,0.1]$; then, let $A_{t+1}=\Pi_{\cS_+^n}(\tilde{A}_t)$ be the projection of $\tilde{A}_t$ onto $\cS_+^n$ such that $A_{t+1}$ is positive semidefinite. The matrices $C_t^{(i)}$, $i=1\ldots,p$ are generated similarly. Let $b_1$ be uniformly distributed from $[-1,1]$. Then, $b_{t+1}=b_t+w_t$ is recursively generated, where $w_t$ is uniformly distributed over $[-0.1,0.1]$. The vectors $d_t^{(i)}$, $i=1,\ldots,p$ are generated in the same way. Finally, each $e_t^{(i)}$ is generated with a particular purpose. After $C_t^{(i)}$ and $d_t^{(i)}$ are generated, we let $h_t^{(i)}$ be uniformly distributed over $[0,1]$ and generate a decision  $\widehat{x}$ with its  entry $\widehat{x}_j$ being uniformly distributed from
$\left(-\frac{R}{\sqrt{n}},\frac{R}{\sqrt{n}}\right)$, then we compute $e_t^{(i)}=-(\frac{1}{2})\widehat{x}^TC_t^{(i)}\widehat{x}+(d_t^{(i)})^T\widehat{x}+h_t^{(i)}$. It is not difficult to verify that in this setting $\widehat{x}$ satisfies the Slater condition (Assumption \ref{assu:slater}). We also set $n=8$, $p=3$, $R=10$, $T=1000$.

We will use Algorithm \ref{alg:MALM} (also denoted by MALM) for solving this numerical case of OQCOP, which is benchmarked by the following two algorithms:
\begin{itemize}
\item CZP. The algorithm (simply denoted by CZP) proposed in \cite{CZP2021} is formed as
\[\left\{
\begin{array}{ll}
\displaystyle x_{t+1}=\Pi_{\cC}[x_t-\eta(\nabla f_{t-\tau}(x_{t-\tau})+J_{t-\tau}(x_{t-\tau})\lambda_{t-\tau})],\\[10pt]
\lambda_{t+1}=[\lambda_t+\eta(g_{t-\tau}(x_{t-\tau})-\delta\eta\lambda_{t-\tau})]_+.
\end{array}
\right.
\]
We set the parameters as $\eta=(\tau T)^{-1/2}$, $\delta=10$ if $\tau\geq 1$; $\eta=T^{-1/2}$, $\delta=10$ if $\tau=0$.
Let us remark that CZP is a simple extension of CL (proposed in \cite{CLiu2019}) for constrained OCO with feedback delays.
\item NY. Inspired by CZP, we can easily extend  algorithm NY proposed in \cite{NYu2017} for solving constrained OCO with feedback delays as follows
\[\left\{
\begin{array}{ll}
\displaystyle x_{t+1}=\Pi_{\cC}\left[x_t-\frac{\nu\nabla f_{t-\tau}(x_{t-\tau})+J_{t-\tau}(x_{t-\tau})\lambda_{t-\tau}}{2\alpha}\right],\\[10pt]
\lambda_{t+1}=[\lambda_t+g_{t-\tau}+J_{t-\tau}(x_{t-\tau})'(x_{t+1}-x_{t-\tau})]_+.
\end{array}
\right.
\]
We denote this algorithm by NY again.
The parameters are set as $\alpha=\tau T$, $\nu=(\tau T)^{1/2}$ if $\tau\geq 1$; $\alpha=T$, $\nu=T^{1/2}$ if $\tau=0$.
\end{itemize}

The numerical results are illustrated in Fig. \ref{fig:oqcqp0}-\ref{fig:oqcqp100} with delays $\tau=0$ (no delay), $\tau=10$, $\tau=20$, $\tau=50$ and $\tau=100$, respectively. From Fig. \ref{fig:oqcqp0} (no delay), we can see the time average regret of MALM is better than CZP and NY, with constraints being not violated during the whole time. From Fig. \ref{fig:oqcqp10}-\ref{fig:oqcqp100} with different delays, we observe that MALM always performs better than the other two algorithms. More interestingly, in theory the bound of constraint violation of CZP is possibly better than that of MALM when the delay is large enough. However, it is shown from Fig. \ref{fig:oqcqp100} with $\tau=100$ that the time average constraint violation of MALM is not worse than CZP. Therefore, it is possible that the theoretical bound of constraint violation with MALM can be further improved.

\begin{figure}[!ht]
\centering
\subfloat[]{\includegraphics[width=2.9in]{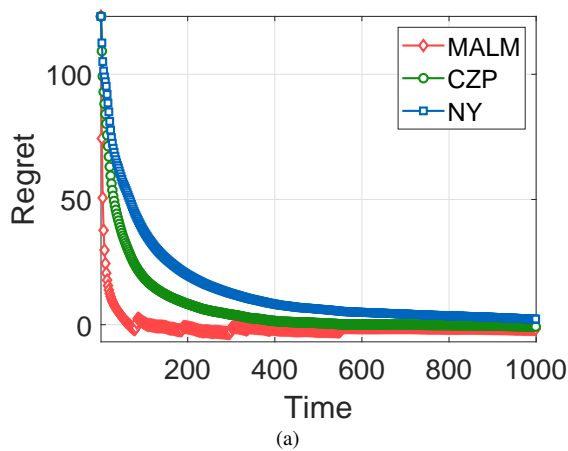}%
\label{oqcqp0-regret}}
\\
\subfloat[]{\includegraphics[width=2.9in]{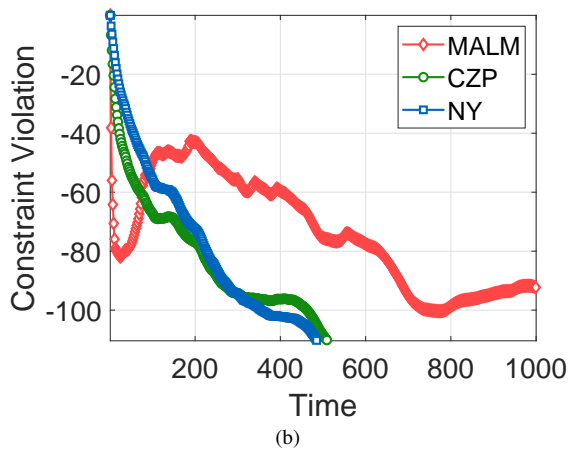}%
\label{oqcqp0-vio}}
\caption{Comparison of algorithms with respect to time average regrets and time average constraint violation for OQCQP with $\tau=0$ (no delay). (a) Time average regrets. (b) Time average constraint violations.}
\label{fig:oqcqp0}
\end{figure}

\begin{figure}[!ht]
\centering
\subfloat[]{\includegraphics[width=2.9in]{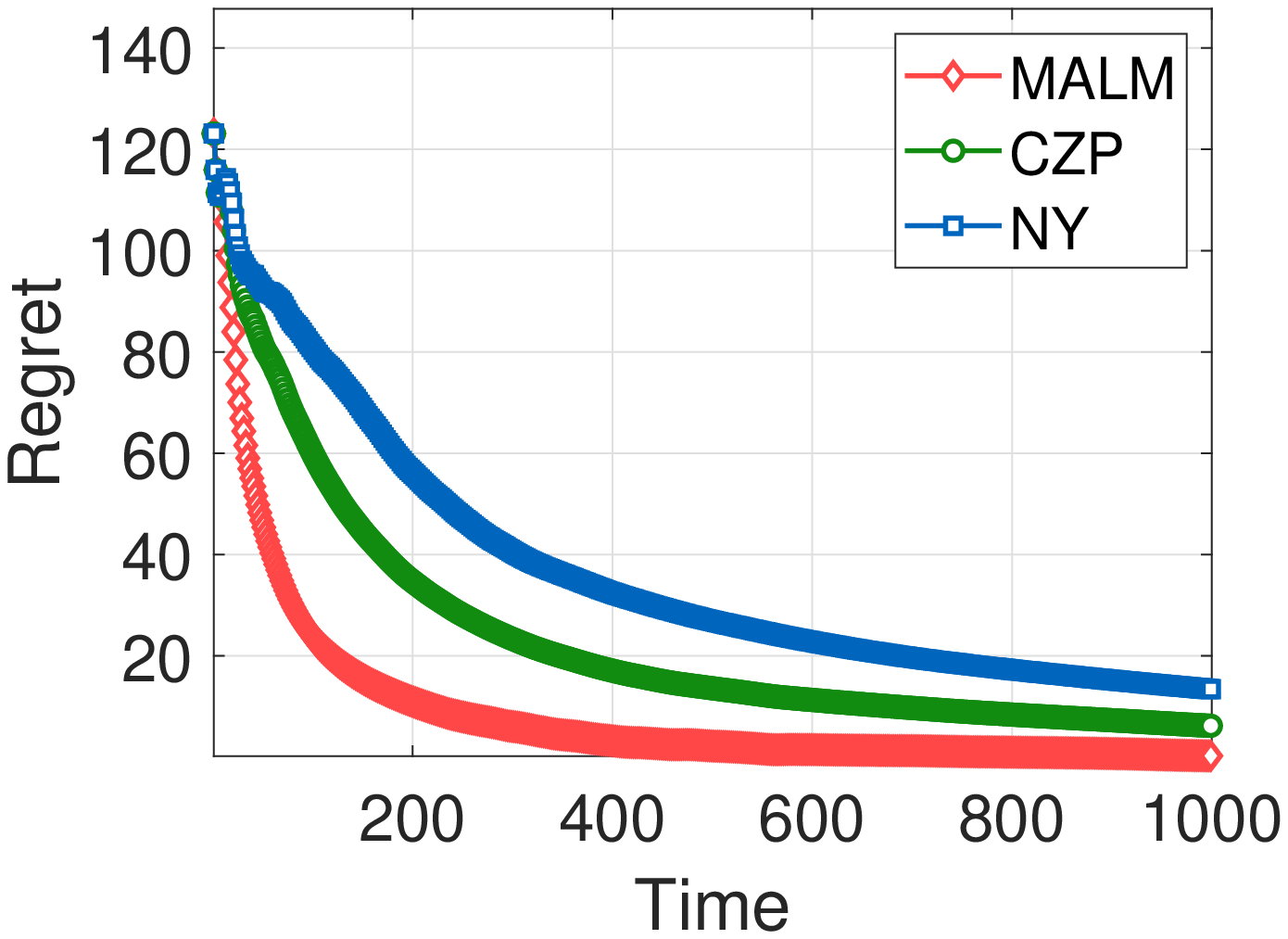}%
\label{oqcqp10-regret}}
\\
\subfloat[]{\includegraphics[width=2.9in]{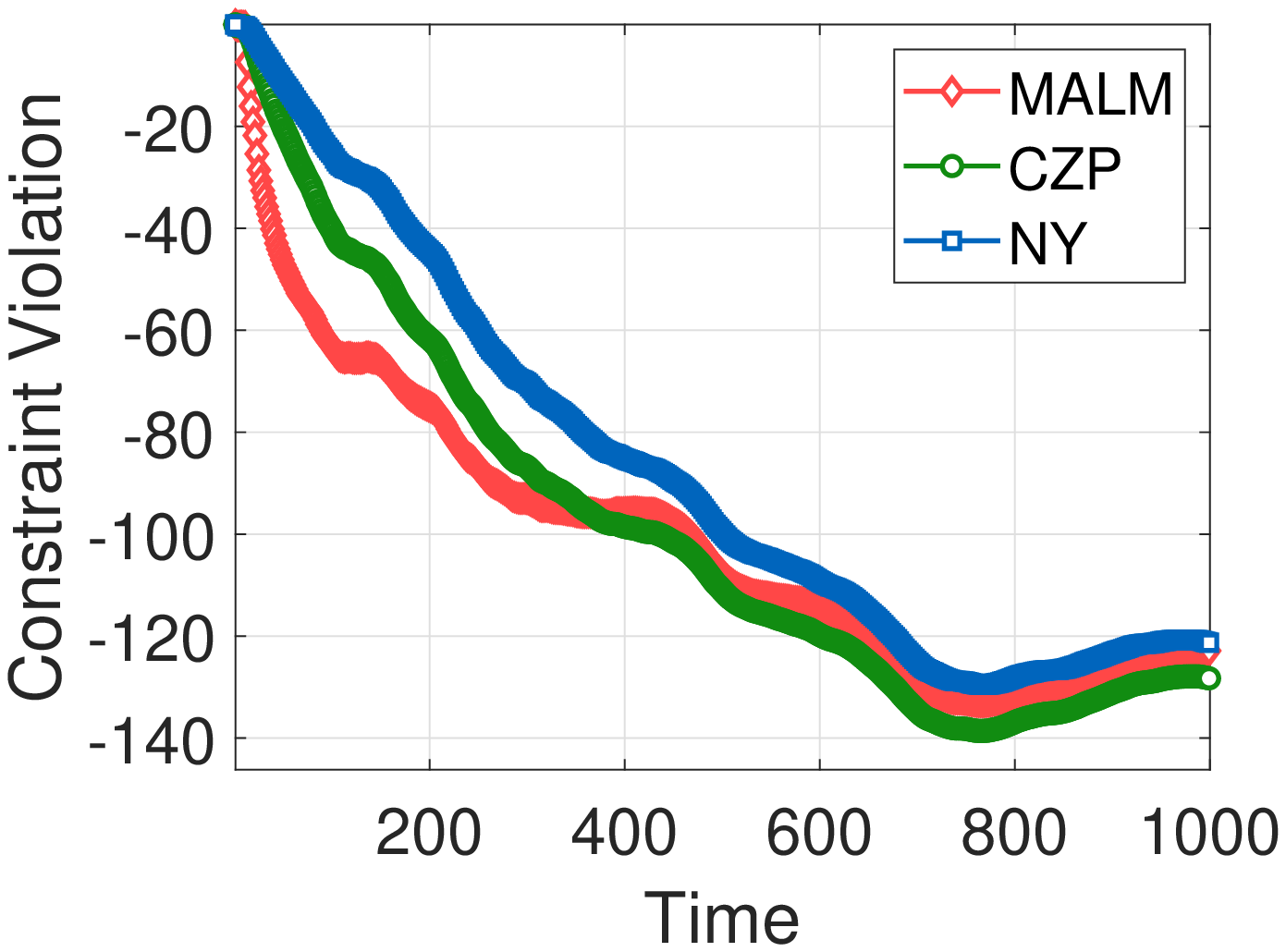}%
\label{oqcqp10-vio}}
\caption{Comparison of algorithms with respect to time average regrets and time average constraint violation for OQCQP with delay $\tau=10$. (a) Time average regrets. (b) Time average constraint violations.}
\label{fig:oqcqp10}
\end{figure}

\begin{figure}[!ht]
\centering
\subfloat[]{\includegraphics[width=2.9in]{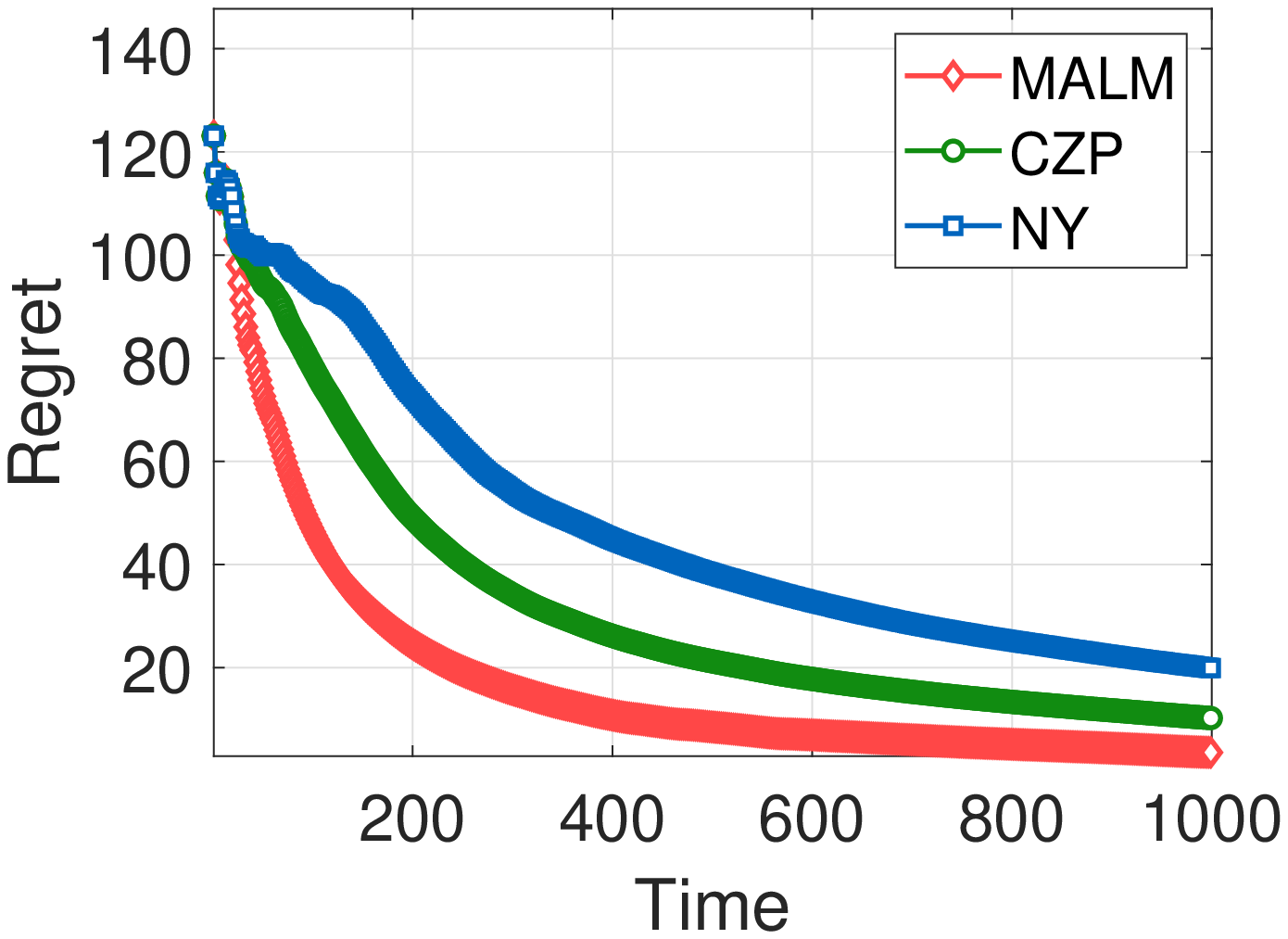}%
\label{oqcqp20-regret}}
\\
\subfloat[]{\includegraphics[width=2.9in]{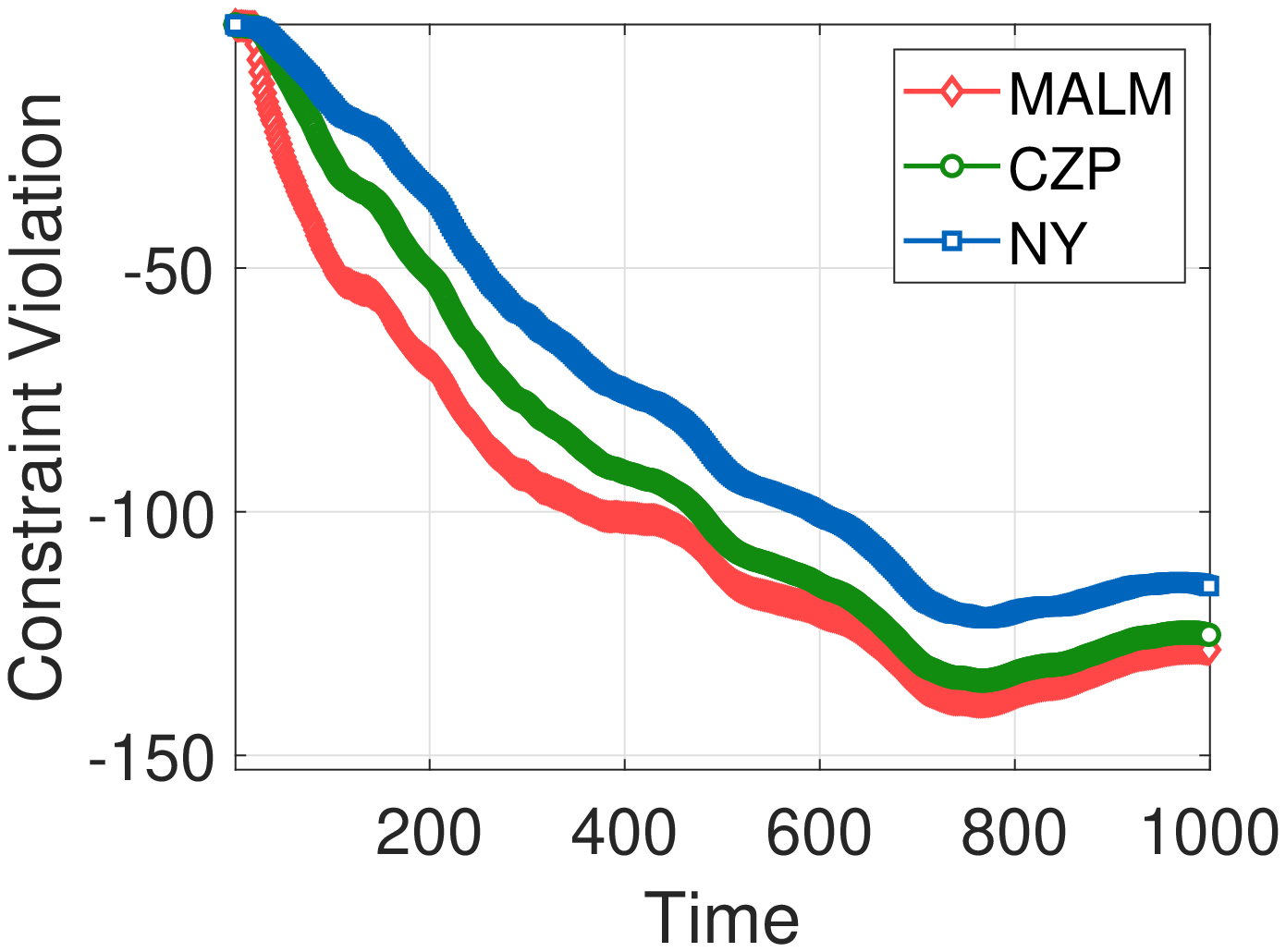}%
\label{oqcqp20-vio}}
\caption{Comparison of algorithms with respect to time average regrets and time average constraint violation for OQCQP with delay $\tau=20$. (a) Time average regrets. (b) Time average constraint violations.}
\label{fig:oqcqp20}
\end{figure}

\begin{figure}[!ht]
\centering
\subfloat[]{\includegraphics[width=2.9in]{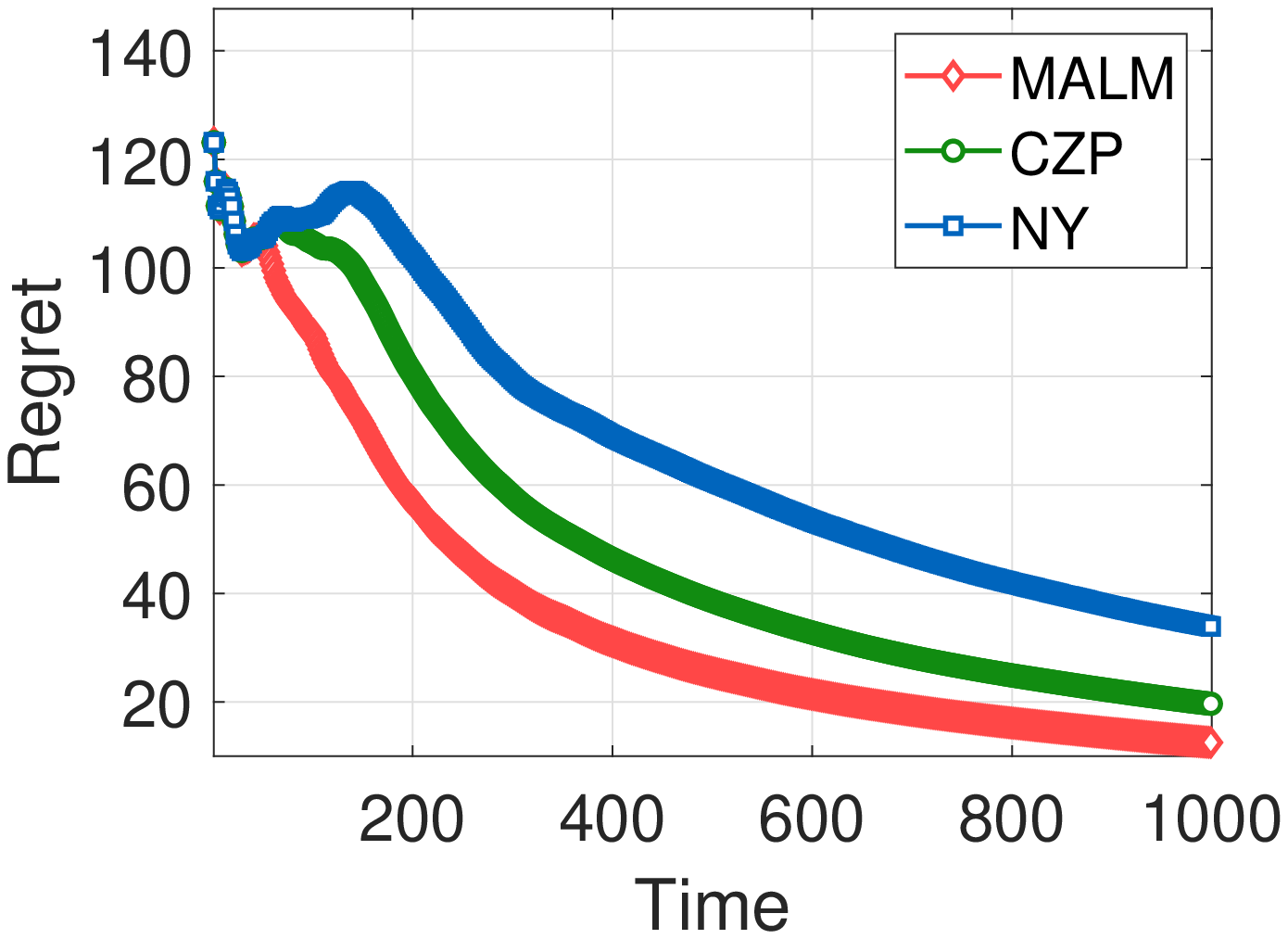}%
\label{oqcqp50-regret}}
\\
\subfloat[]{\includegraphics[width=2.9in]{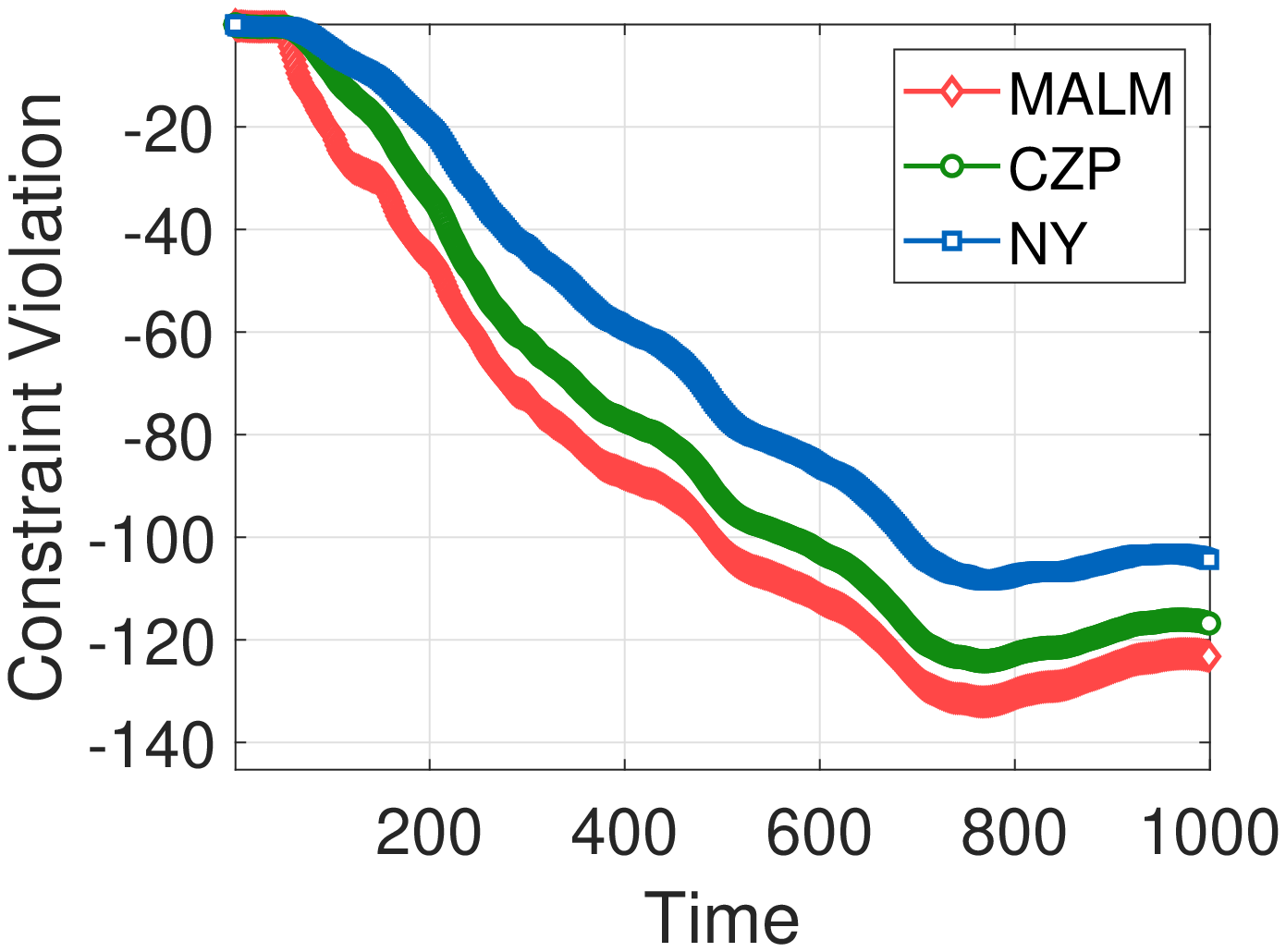}%
\label{oqcqp50-vio}}
\caption{Comparison of algorithms with respect to time average regrets and time average constraint violation for OQCQP with delay $\tau=50$. (a) Time average regrets. (b) Time average constraint violations.}
\label{fig:oqcqp50}
\end{figure}

\begin{figure}[!ht]
\centering
\subfloat[]{\includegraphics[width=2.9in]{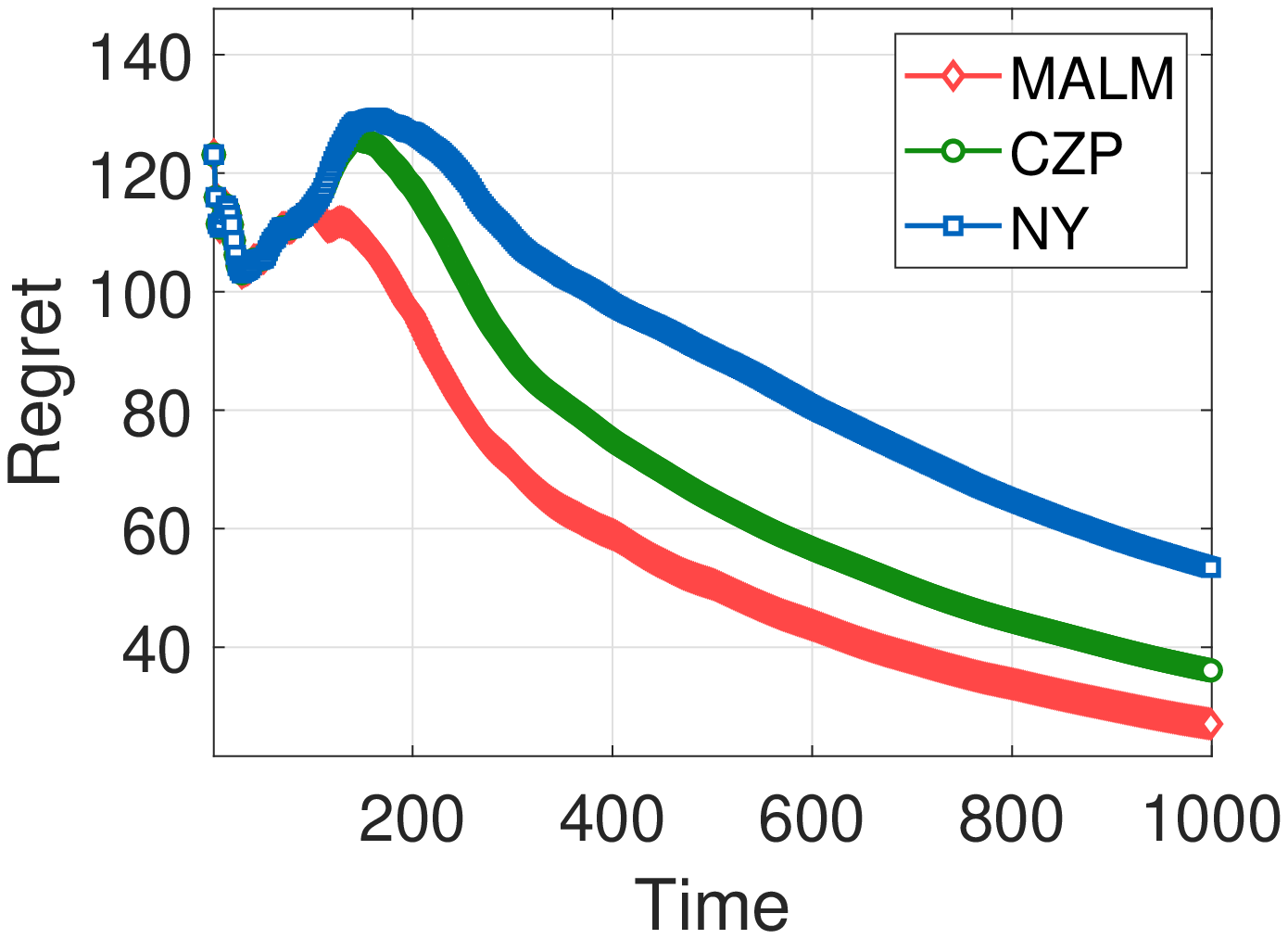}%
\label{oqcqp100-regret}}
\\
\subfloat[]{\includegraphics[width=2.9in]{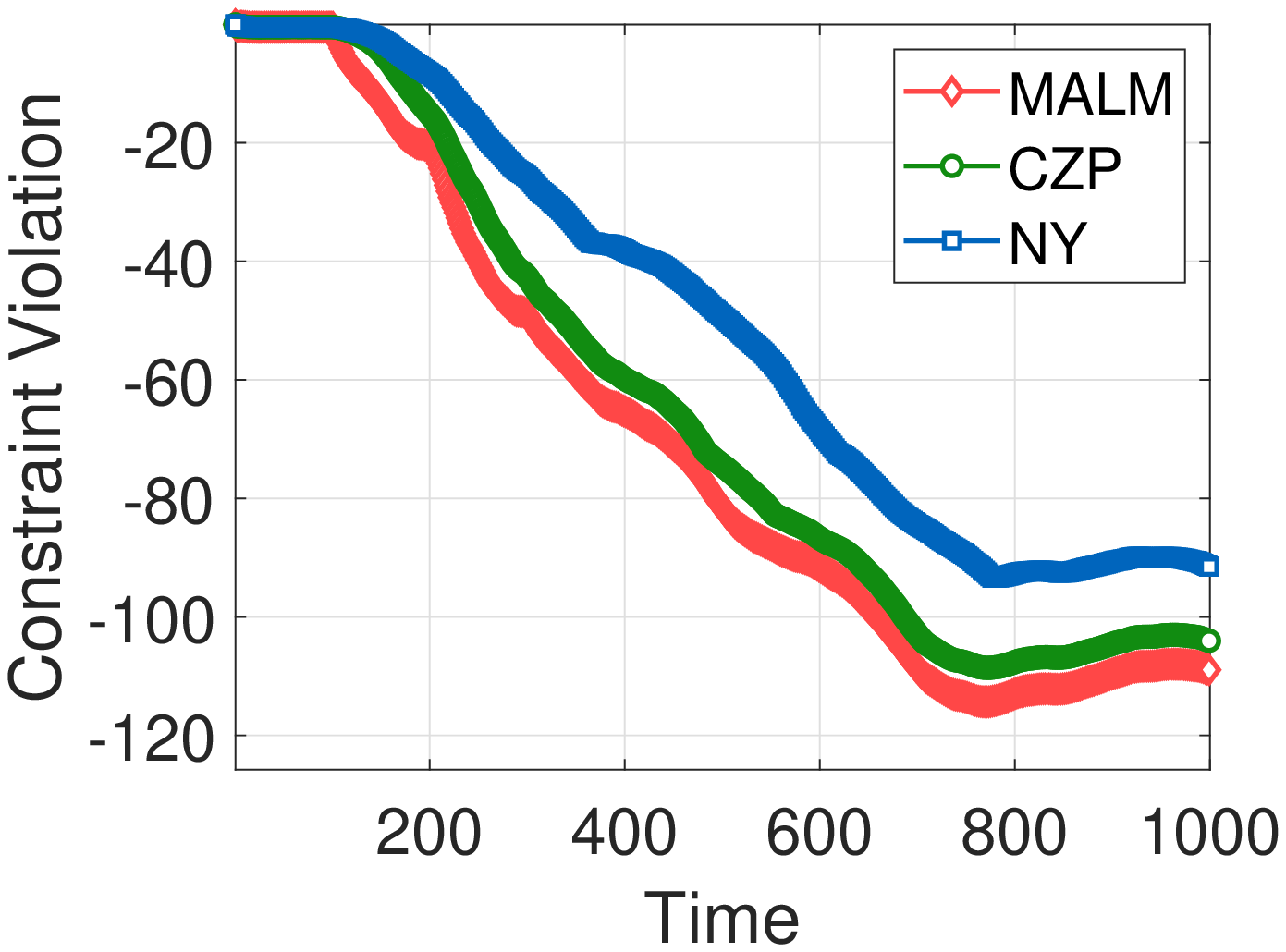}%
\label{oqcqp100-vio}}
\caption{Comparison of algorithms with respect to time average regrets and time average constraint violation for OQCQP with delay $\tau=100$. (a) Time average regrets. (b) Time average constraint violations.}
\label{fig:oqcqp100}
\end{figure}
\section{Conclusion}\label{sec:conclude}
In this paper, a class of model-based augmented Lagrangian methods for time-varying constrained OCO without/with feedback delays have been studied. For both two types of constrained OCO, we have established the sublinear regret and sublinear constraint violation bounds. Various numerical examples have been presented to verify the efficiency of the proposed algorithms.

The proposed algorithms are possible to extend in several interesting research directions. For instance, in some applications of OCO, the (sub)gradients of the loss or/and constraint functions are difficult to obtain. Instead, only function values are available sequentially to the decision maker, which is usually called \textit{bandit feedback} \cite{CLiu2019,CZP2021}. Another interesting line of research is OCO with stochastic functional constraints that are i.i.d. generated at each time slot \cite{YMNeely2017,CZP2021b}.

\appendix  
\section* {An useful lemma}
The following lemma is from \cite[Lemma 6]{ZLX2022}.
\begin{lemma}\label{lem:l7}
Let $\{Z_t\}$ be a sequence  with $Z_0 = 0$. Suppose there exist an integer $t_0 >0$, real constants $\theta>0$, $\delta_{\max}>0$ and $ 0 <\zeta \leq \delta_{\max}$ such that
$\lvert Z_{t+1}-Z_t\rvert  \leq \delta_{\max}$
and
\[
Z_{t+t_0}-Z_t  \leq
-t_0\zeta, \quad \mbox{if } Z_t \geq \theta
\]
hold for all $t \in \{1,2,\ldots\}.$ Then,
\[
Z_t \leq \theta +t_0 \delta_{\max}+t_0  \frac{4 \delta_{\max}^2}{\zeta}\log \left[  \frac{8 \delta_{\max}^2}{\zeta^2} \right], \forall t \in \{1,2,\ldots\}.
\]
\end{lemma}
%




\ifCLASSOPTIONcaptionsoff
  \newpage
\fi



%
\bibliography{ref.bib}{}

\begin{thebibliography}{10}
\providecommand{\url}[1]{#1}
\csname url@samestyle\endcsname
\providecommand{\newblock}{\relax}
\providecommand{\bibinfo}[2]{#2}
\providecommand{\BIBentrySTDinterwordspacing}{\spaceskip=0pt\relax}
\providecommand{\BIBentryALTinterwordstretchfactor}{4}
\providecommand{\BIBentryALTinterwordspacing}{\spaceskip=\fontdimen2\font plus
\BIBentryALTinterwordstretchfactor\fontdimen3\font minus
  \fontdimen4\font\relax}
\providecommand{\BIBforeignlanguage}[2]{{%
\expandafter\ifx\csname l@#1\endcsname\relax
\typeout{** WARNING: IEEEtran.bst: No hyphenation pattern has been}%
\typeout{** loaded for the language `#1'. Using the pattern for}%
\typeout{** the default language instead.}%
\else
\language=\csname l@#1\endcsname
\fi
#2}}
\providecommand{\BIBdecl}{\relax}
\BIBdecl

\bibitem{Shai2011}
S.~Shalev-Shwartz, ``Online learning and online convex optimization,''
  \emph{Foundations and Trends{\textregistered} in Machine Learning}, vol.~4,
  no.~2, pp. 107--194, 2011.

\bibitem{Hazan2015}
E.~Hazan, ``Introduction to online convex optimization,'' \emph{Foundations and
  Trends{\textregistered} in Optimization}, vol.~2, no. 3-4, pp. 157--325,
  2015.

\bibitem{HSLZ2021}
S.~C. Hoi, D.~Sahoo, J.~Lu, and P.~Zhao, ``Online learning: A comprehensive
  survey,'' \emph{Neurocomputing}, vol. 459, pp. 249--289, 2021.

\bibitem{Zi2003}
M.~Zinkevich, ``Online convex programming and generalized infinitesimal
  gradient ascent,'' in \emph{Proceedings of the 20th International Conference
  on Machine Learning}, 2003, pp. 928--935.

\bibitem{HAK2007}
E.~Hazan, A.~Agarwal, and S.~Kale, ``Logarithmic regret algorithms for online
  convex optimization,'' \emph{Mach. Learn.}, vol.~69, pp. 169--192, 2007.

\bibitem{LTS2021}
A.~Lesage-Landry, J.~A. Taylor, and I.~Shames, ``Second-order online nonconvex
  optimization,'' \emph{IEEE Trans. Automat. Control}, vol.~66, no.~10, pp.
  4866--4872, 2021.

\bibitem{MRYang2012}
M.~Mahdavi, R.~Jin, and T.~Yang, ``Trading regret for efficiency: online convex
  optimization with long term constraints,'' \emph{J. Mach. Learn. Res.},
  vol.~13, pp. 2503--2528, 2012.

\bibitem{YN2020}
H.~Yu and M.~J. Neely, ``A low complexity algorithm with {$O(\sqrt T)$} regret
  and {$O(1)$} constraint violations for online convex optimization with long
  term constraints,'' \emph{J. Mach. Learn. Res.}, vol.~21, pp. Paper No. 1,
  24, 2020.

\bibitem{MTY2009}
S.~Mannor, J.~N. Tsitsiklis, and J.~Y. Yu, ``Online learning with sample path
  constraints,'' \emph{J. Mach. Learn. Res.}, vol.~10, pp. 569--590, 2009.

\bibitem{PR2017}
S.~Paternain and A.~Ribeiro, ``Online learning of feasible strategies in
  unknown environments,'' \emph{IEEE Trans. Automat. Control}, vol.~62, no.~6,
  pp. 2807--2822, 2017.

\bibitem{CLG2017}
T.~Chen, Q.~Ling, and G.~B. Giannakis, ``An online convex optimization approach
  to proactive network resource allocation,'' \emph{IEEE Trans. Signal
  Process.}, vol.~65, no.~24, pp. 6350--6364, 2017.

\bibitem{NYu2017}
M.~J. Neely and H.~Yu, ``Online convex optimization with time-varying
  constraints,'' 02 2017, https://arxiv.org/abs/1702.04783.

\bibitem{CLiu2019}
X.~Cao and K.~J.~R. Liu, ``Online convex optimization with time-varying
  constraints and bandit feedback,'' \emph{IEEE Trans. Automat. Control},
  vol.~64, no.~7, pp. 2665--2680, 2019.

\bibitem{ZDH2021}
Y.~Zhang, E.~Dall'Anese, and M.~Hong, ``Online proximal-{ADMM} for time-varying
  constrained convex optimization,'' \emph{IEEE Trans. Signal Inform. Process.
  Netw.}, vol.~7, pp. 144--155, 2021.

\bibitem{FPPR2018}
M.~Fazlyab, S.~Paternain, V.~M. Preciado, and A.~Ribeiro,
  ``Prediction-correction interior-point method for time-varying convex
  optimization,'' \emph{IEEE Trans. Automat. Control}, vol.~63, no.~7, pp.
  1973--1986, 2018.

\bibitem{CB2021a}
X.~Cao and T.~Ba\c{s}ar, ``Decentralized online convex optimization based on
  signs of relative states,'' \emph{Automatica J. IFAC}, vol. 129, pp. Paper
  No. 109\,676, 13, 2021.

\bibitem{Cb2021b}
------, ``Decentralized online convex optimization with event-triggered
  communications,'' \emph{IEEE Trans. Signal Process.}, vol.~69, pp. 284--299,
  2021.

\bibitem{YLYXCJ2021}
X.~Yi, X.~Li, T.~Yang, L.~Xie, T.~Chai, and K.~H. Johansson, ``Distributed
  bandit online convex optimization with time-varying coupled inequality
  constraints,'' \emph{IEEE Trans. Automat. Control}, vol.~66, no.~10, pp.
  4620--4635, 2021.

\bibitem{Rockafellar76B}
R.~T. Rockafellar, ``Augmented {L}agrangians and applications of the proximal
  point algorithm in convex programming,'' \emph{Math. Oper. Res.}, vol.~1,
  no.~2, pp. 97--116, 1976.

\bibitem{ZST2010}
X.-Y. Zhao, D.~Sun, and K.-C. Toh, ``A {N}ewton-{CG} augmented {L}agrangian
  method for semidefinite programming,'' \emph{SIAM J. Optim.}, vol.~20, no.~4,
  pp. 1737--1765, 2010.

\bibitem{CDZ2015}
N.~Chatzipanagiotis, D.~Dentcheva, and M.~M. Zavlanos, ``An augmented
  {L}agrangian method for distributed optimization,'' \emph{Math. Program.},
  vol. 152, no. 1-2, Ser. A, pp. 405--434, 2015.

\bibitem{CZ2016}
N.~Chatzipanagiotis and M.~M. Zavlanos, ``A distributed algorithm for convex
  constrained optimization under noise,'' \emph{IEEE Trans. Automat. Control},
  vol.~61, no.~9, pp. 2496--2511, 2016.

\bibitem{MK2021}
H.~Moradian and S.~S. Kia, ``Cluster-based distributed augmented {L}agrangian
  algorithm for a class of constrained convex optimization problems,''
  \emph{Automatica J. IFAC}, vol. 129, pp. Paper No. 109\,608, 8, 2021.

\bibitem{ZZWX2022}
L.~Zhang, Y.~Zhang, J.~Wu, and X.~Xiao, ``Solving stochastic optimization with
  expectation constraints efficiently by a stochastic augmented lagrangian-type
  algorithm,'' 3 2022, https://arxiv.org/abs/2106.11577.

\bibitem{ZLS2009}
M.~Zinkevich, J.~Langford, and A.~Smola, ``Slow learners are fast,'' in
  \emph{Advances in Neural Information Processing Systems}, vol.~22, 2009.

\bibitem{JGS2013}
P.~Joulani, A.~Gyorgy, and C.~Szepesvari, ``Online learning under delayed
  feedback,'' in \emph{Proceedings of the 30th International Conference on
  Machine Learning}, vol.~28, no.~3.\hskip 1em plus 0.5em minus 0.4em\relax
  PMLR, 2013, pp. 1453--1461.

\bibitem{PSLW2022}
W.~Pan, G.~Shi, Y.~Lin, and A.~Wierman, ``Online optimization with feedback
  delay and nonlinear switching cost,'' \emph{Proc. ACM Meas. Anal. Comput.
  Syst.}, vol.~6, no.~1, pp. 1--34, 2022.

\bibitem{CZP2021}
X.~Cao, J.~Zhang, and H.~V. Poor, ``Constrained online convex optimization with
  feedback delays,'' \emph{IEEE Trans. Automat. Control}, vol.~66, no.~11, pp.
  5049--5064, 2021.

\bibitem{AD2019I}
H.~Asi and J.~C. Duchi, ``Stochastic (approximate) proximal point methods:
  convergence, optimality, and adaptivity,'' \emph{SIAM J. Optim.}, vol.~29,
  no.~3, pp. 2257--2290, 2019.

\bibitem{LDPSM2019}
N.~Liakopoulos, A.~Destounis, G.~Paschos, T.~Spyropoulos, and P.~Mertikopoulos,
  ``Cautious regret minimization: Online optimization with long-term budget
  constraints,'' in \emph{Proceedings of the 36th International Conference on
  Machine Learning}, vol.~97, 09--15 Jun 2019, pp. 3944--3952.

\bibitem{BDS2019}
A.~Bernstein, E.~Dall'Anese, and A.~Simonetto, ``Online primal-dual methods
  with measurement feedback for time-varying convex optimization,'' \emph{IEEE
  Trans. Signal Process.}, vol.~67, no.~8, pp. 1978--1991, 2019.

\bibitem{cvx}
M.~Grant and S.~Boyd, ``{CVX}: Matlab software for disciplined convex
  programming, version 2.1,'' \url{http://cvxr.com/cvx}, Mar. 2014.

\bibitem{YMNeely2017}
H.~Yu, M.~J. Neely, and X.~Wei, ``Online convex optimization with stochastic
  constraints,'' in \emph{Advances in Neural Information Processing Systems},
  2017, pp. 1428--1438.

\bibitem{CZP2021b}
X.~Cao, J.~Zhang, and H.~V. Poor, ``Online stochastic optimization with
  time-varying distributions,'' \emph{IEEE Trans. Automat. Control}, vol.~66,
  no.~4, pp. 1840--1847, 2021.

\bibitem{ZLX2022}
L.~Zhang, H.~Liu, and X.~Xiao, ``Regrets of proximal method of multipliers for
  online non-convex optimization with long term constraints,'' 04 2022,
  https://arxiv.org/abs/2204.10986.

\end{thebibliography}
\bibliographystyle{IEEEtran}

%

%
\begin{IEEEbiographynophoto}{Haoyang Liu}
received the B.S. degree in Applied Mathematics from Dalian University of Technology, the M.S. degree in Commerce Finance from University of New South Wales, in 2011 and 2017, respectively.  He is currently a Ph.D. student in  the School of Mathematical Sciences at Dalian University of Technology, Dalian, China.
His research interest lies in the area of optimization, online learning and financial portfolio.
His current research interest is focused on the development of numerical methods in online learning and online convex optimization.
\end{IEEEbiographynophoto}
%
%
\begin{IEEEbiographynophoto}{Xiantao Xiao}
received the B.S. degree from the Zhengzhou University, Zhengzhou, China, in 2003. He received the Ph.D. degree  in operations research and control theory from Dalian University of Technology, Dalian, China, in 2009.
He has been a professor in School of Mathematical Sciences at Dalian University of Technology since Winter, 2021. He joined the School of Mathematical Sciences at Dalian University of Technology since July 2009. His research interest lies in the development of theory and methods in the context of mathematical optimization, machine learning and stochastic approximation.
\end{IEEEbiographynophoto}

\begin{IEEEbiographynophoto}{Liwei Zhang}
received the B.S. degree, the M.S. degree and Ph.D. degree, all in Applied Mathematics from Dalian University of Technology, in 1989, 1992 and 1998, respectively.
He is now a professor in School of Mathematical Sciences at Dalian University of Technology. His current research interest is focused on stochastic optimization, matrix optimization and online learning.
\end{IEEEbiographynophoto}




\end{document}